\documentclass{amsart}
\usepackage{amssymb,amsmath,amsthm,amsfonts, mathrsfs, mathtools}
\usepackage{graphicx} 

\usepackage[backgroundcolor=green!40, bordercolor = white, linecolor=green!40, colorinlistoftodos, textsize=footnotesize]{todonotes}
\usepackage[english]{babel}
\usepackage[latin1]{inputenc}
\usepackage[sc]{mathpazo}
\usepackage[colorlinks]{hyperref}

\usepackage{tikz}
\usetikzlibrary{patterns}
\usetikzlibrary {patterns,patterns.meta}
\usetikzlibrary{shapes.misc}
\tikzset{cross/.style={cross out, draw=black, minimum size=2*(#1-\pgflinewidth),inner sep=0pt,outer sep=0pt},cross/.default={3pt}}

\def\R {\mathbb{R}}

\def\L {\mathcal{L}}

\def\sign{\mathrm{sign}}

\renewcommand{\div}{\mathrm{div}}

\def\na{\nabla}
\def\pa{\partial}
\def\Lip{\mathrm {Lip}}
\def\Supp{\mathrm {Supp}}

\def\supp{\mathrm{Supp}\, }

\newtheorem{proposition}{Proposition}[section]
\newtheorem{theorem}[proposition]{Theorem}
\newtheorem*{theorem*}{Theorem}
\newtheorem{corollary}[proposition]{Corollary}
\newtheorem{lemma}[proposition]{Lemma}

\theoremstyle{definition}
\newtheorem{definition}[proposition]{Definition}
\newtheorem{remark}[proposition]{Remark}
\numberwithin{equation}{section}

\newcommand{\qtext}[1]{\quad\text{#1}\quad}

\newcommand{\hb}{}

\title{
Range expansion by growth and congestion
}
\author{Henri Berestycki}
\address{H. Berestycki: Department of Mathematics,
University of Maryland, College Park, USA; Centre d'analyse et math\'ematique sociales, \'Ecole des hautes \'etudes en sciences sociales, Paris, France;
Institute for Advanced Study,
Hong Kong University of Science and Technology, Hong Kong}
\email{hberes@umd.edu}

\author{Antoine Mellet}
\address{A. Mellet: Department of Mathematics, University of Maryland-College Park, USA} \email{mellet@umd.edu}

\thanks{A. Mellet was partially supported by NSF Grant DMS-2307342.}

\date{}



\begin{document}

\begin{abstract}
We introduce here a nonlinear and nonlocal model that describes the range expansion of a population resulting from growth and competition for space.  
This type of phenomenon underlies the expansion of colonies of immotile cells which motivated this work. These colonies display long-range shuffling induced by congestion, and are also subject to jamming. 
We start by showing the well-posedness of the general evolution equation.
We then derive a singular limit of this model corresponding to a regime where dispersal occurs only from saturated areas. 
The limiting model, which has the structure of an obstacle free boundary problem in time, provides an effective approach to the description of the range expansion of a population as a result of growth, saturation and dispersion.
We  establish the main mathematical properties of this singular problem by proving a comparison property and describing the dynamics of the free boundary that delimits the saturated area. We identify traveling wave solutions and characterize the asymptotic speed of spreading of compactly supported solutions. 
\end{abstract}

\maketitle

\tableofcontents

\section*{Introduction}

Reaction-diffusion equations, with local and non-local diffusion, provide the classical modeling framework for the spread of biological populations, capturing essential features such as growth, competition for resources and range expansion. Although these models are remarkably successful in various contexts, their derivation typically relies on assumptions of random motion of individuals in this population.
Brownian motion or jump processes type assumptions lead to  linear diffusion 
(as in the classical Fisher-KPP equation \cite{KPP1937,Fisher1937}),
while fluid-like behavior yields nonlinear diffusion models
(commonly used for tumor growth \cite{Byrne1996,Bru2003,byrne09}).

However, in many instances, this type of assumption of random movement is not warranted. Often, individual cells do not move on their own, yet the growth process still results in colony expansion owing to shuffling caused by congestion quantified by a pressure. Hence, growth and competition for space are core mechanisms of the range expansion.  
This is the case for the range expansion of immotile cells \cite{Delarue2016,Warren2019} which is the main motivation for this work.
These cells, by definition, do not actively move around; instead, population expansion occurs through local proliferation (mitosis) and mechanical interactions between cells (pushing).  

Moreover, reaction-diffusion equations do not agree with some characteristic features of the expansion of such colonies. 
For example, solutions of reaction-diffusion equations are not typically monotone increasing in time (at least not initially),
but experimental observations show that immotile cells only move away from their position when they are being replaced by other cells, thereby suggesting that time monotonicity should be  a  feature of a mathematical model describing this phenomena.


\medskip

In this paper, we introduce a new  nonlinear and nonlocal equation for range expansion of a cell colony.
It describes the evolution of a population that grows by the cell division mechanism. At low density, the growth is purely local until congestion (which corresponds to the lack of available space rather than diminishing resources) leads to relocation   
of individuals at a nearby non fully congested space.
Cells moving to occupy a new location hitherto empty are not necessarily new cells created by duplication but their movement is caused by the new cells pushing  nearby cells to create space. Thus, the model captures the expansion of the spatial reach of the population resulting from growth and shuffling.
While  preserving many features of reaction-diffusion equations, this new equation also differs on some key aspects, for instance by exhibiting monotonicity in time at every point.

This is a continuum model, and, as such, is only relevant when the density approximation is warranted. Therefore, it cannot capture phenomena at the scale of individual cells. In particular, it does not offer a satisfactory description of an initial phase with only a handful of cells around. Yet, by aiming at representing a population of cells tracking the duplications and positions of individual cells, it can be considered to be of mesoscopic type. 

After introducing the general model (see Section \ref{derivation0} below), we establish some of its basic mathematical properties in Section \ref{properties}. 
We then investigate a singular limit that highlights the role of saturation in colony expansion. 
This limit is analogous to the stiff pressure limit of the porous medium equation, a classical framework for congestion particularly in tumor growth (see for example \cite{PQV,MPQ,Perthame2014,DP}). The resulting asymptotic model takes the form of a free boundary problem (obstacle problem type):
\begin{equation} \label{eq:main0} \tag{$\star$}
\begin{cases}
\pa_t u  =  g(u) (1-K*{\mathbb 1}_{\{u=1\}}) + g(1) K*{\mathbb 1}_{\{u=1\}}
& \mbox{ when } u<1, \\
\pa_t u  =  0
& \mbox{ when } u=1.
\end{cases}
\end{equation} 

Here, ${\mathbb 1}_{\{u=1\}}$ stands for the indicator function of the saturated set  ${\{u=1\}}$, that is 
${\mathbb 1}_{\{u=1\}}(x)=1$ when $u(x)=1$ and ${\mathbb 1}_{\{u=1\}}(x)=0$ if $u(x)<1$. 
 We emphasize that this limit differs from the Hele-Shaw type models derived for example in \cite{PQV} and often associated to the modeling of biological tissues growth. As we will see, this new model is significantly more amenable to numerical simulation and has important features that are in qualitative agreement with experimental observations.
The main mathematical developments in this paper concern this singular equation. We establish a comparison principle, prove the existence of traveling fronts and study the asymptotic speed of spreading for compactly supported initial data. 

Although this work was motivated by the study of immotile cell colonies, the type of mechanism described by this singular model is fairly general and is relevant to other frameworks as well. 
In particular, as we discuss briefly in Section \ref{sprawl}, this type of model could be used to describe the growth of cities \cite{MB25}.

 \medskip

{\bf Organization of the paper:} 
In Section~\ref{derivation}, we introduce the new general framework and discuss its motivation stemming from the range expansion of immotile cells colonies. We also
formally derive the above free boundary problem as a singular pressure limit of the general model. 
We then state the main mathematical properties concerning this singular pressure model 
in Section~\ref{sec:results}. In section~\ref{sec:compliterature}, we compare this new model to the classical approaches describing biological invasions. 
We rigorously justify the singular limit in Section~\ref{sec:limit} and establish a comparison principle in Section~\ref{sec:comparison}. Invasion properties and traveling waves are analyzed in Section~\ref{sec:TW}. They allow us to establish the asymptotic speed of spreading obtained by a linear determinacy in Section~\ref{sec:spreading}.

\medskip

{\bf Acknowledgments:} The authors are grateful to Gabi Steinbach for many inspiring and fruitful discussions during the preparation of this manuscript. 
The mathematical developments we present here were motivated by some of Dr. Steinbach's experimental observations.
In forthcoming joint work~\cite{BMS2}, we further discuss the biological motivation for this work, in particular in connection with experimental observations and the derivation of the model.

The authors are also thankful to Yoichiro Mori for his insightful remarks on an earlier version of this work.

\section{The general model}\label{derivation}
\subsection{A nonlinear nonlocal  equation for range expansion}\label{derivation0}
We introduce a model which describes the range expansion of a colony of immotile cells, resulting from growth and pressure mediated dispersal. 
Denoting the cell population density  by $u(t,x)$  (at time $t>0$ and  spatial location $x\in \R^d$), this model writes:
\begin{equation}\label{eq:main1}
\pa_t u  = F[u] +\L[u],
\end{equation}
with $F[u]$ representing local growth and $\L[u]$ modeling dispersal. However, both terms are non-local. 
Introducing a pressure $p$ and a dispersal kernel $K$, we take these operators to be of the form
\begin{equation}\label{eq:F}
F[u]=g(u) (1-K*p)
\end{equation}
and
\begin{equation}\label{eq:L}
\mathcal L[u]:=\int_{\R^d} K(x-y)\Big[ g(u(y)) (1-p(x)) p(y) - g(u(x)) (1-p(y)) p(x) \Big] \, dy.
\end{equation}
The pressure variable is given by $p(t,x)=P(u(t,x))$ with $P$ satisfying
$$ \mbox{ $u \mapsto P(u)$  is increasing  with $P(0)=0$ and $P(1)=1$}.$$
The kernel $K=K(|x-y|)\geq 0$ is a function of the distance and is normalized:
\begin{equation}\label{eq:Knorm}
\int_{\R^d} K(x)\, dx  = 1 .
\end{equation}

Let us now explain the meaning of the various terms in \eqref{eq:main1}. First, the role of the pressure term is to represent the effect of congestion. The maximum pressure is normalized to be $1$ when $u$ reaches its saturation level, $u=1$, which can be interpreted as a maximum packing density. 

The  function $F[u]$ describes the local production of new cells owing to the growth process (e.g. mitosis). 
The growth function $g\geq 0$, typically given by $g(u)=ru$ (or $g(u)=\hb{r_0}u(1-\frac u M)$ if depletion of resources is relevant), is the rate of growth in the absence of competition for space.
The growth process is limited by congestion in a nonlocal way. Indeed, even if a location $x$ is saturated, mitosis can still occur there, provided there is some free space at a distance that can be reached through pushing of cells. Full congestion occurs at $x$ if a neighborhood is completely saturated. 
The size of the relevant neighborhood involves a characteristic length for shuffling, embodied in the kernel $K$.
\smallskip

Unlike reaction-diffusion equations where dispersion is described by the standard diffusion operator, Equation \eqref{eq:main1} includes a nonlocal operator $\L$ that captures the cellular rearrangements induced by growth and competition for space.
This integral operator is conservative, that is:
$$\int_{\R^d} \mathcal L[u](x)\, dx=0 \qquad \mbox{ for any function } u(x).$$ 
Hence, it does not contribute to growth.
It can be split into gain and loss terms
$\mathcal L[u] = \mathcal L_+[u]-\mathcal L_-[u]$. The loss term, given by
\begin{equation}\label{eq:loss}
\L_-[u](x)=\int_{\R^d} K(x-y)g(u(x)) (1-p(y)) p(x) \, dy =F[u](x) p(x),
\end{equation}
accounts for the fact that some fraction of the newly produced cells at position $x$ do not contribute to the increase of $u(t,x)$. 
Indeed, cell division may give rise to some mechanical displacement of pre-existing cells, leading to a local rearrangement of the cell population.
As a consequence, some cells are pushed to occupy new locations farther from the site of mitosis. 
The form of the loss term \eqref{eq:loss} captures the dependence of this process on the pressure $p$. In particular, at saturation, when $p(t,x)=1$, the terms
$F[u](x)$ and $\L_-[u](x)$ balance each other. Indeed, then, no new cells can be placed at location $x$.
More generally, the closer $p$ is to 1 around $x$, the farther cells away from the location $x$  are shuffled and the more constrained the duplication is. It comes to a halt when  the vicinity of $x$ is fully saturated. This corresponds to the mechanism of {\em growth arrest} by contact inhibition.

The gain term $\mathcal L_+$ (obtained from $\L_-$ by symmetry considerations) describes the transfer of new cells to $x$ following production at nearby locations $y$.

It is worth emphasizing that dispersal only arises as a consequence from growth in contrast with classical modeling that assumes movement of all existing cells. This is akin in spirit to some models of biological invasions introduced by Kot, Lewis and van den Driessche\cite{KotLewisVDD1996} that involve long range dispersal of new individuals as for instance seeds. The model we develop here has the additional complexity that the new cells involve displacements of other cells as well.
\medskip

We note that the interaction kernel $K(x-y)$, which appears in both $F$ and $\L$ describes a complex mechanism which involves the pre-mitosis growth of a cell and the mechanical forces exerted by the growing cell on its neighbors.
Because of friction and other viscous forces, these interactions cannot occur over long distances and we will assume that $K$ is supported in a ball  $B_\ell (0)$.

We refer the reader to our work with G. Steinbach \cite{BMS2}, in which we give a more detailed discussion of the biological grounds that justifies this model \eqref{eq:main1}.

\medskip

\subsection{Simple properties of \eqref{eq:main1}}\label{properties}
Equation \eqref{eq:main1} introduces a novel phenomenological framework for modeling the evolution of population density when
 growth and congestion resulting in shuffling are the main drivers of range expansion.
A remarkable property of this equation is that its solutions are 
monotone increasing functions of time which satisfy $0\leq u(t,x)\leq 1$ for all $t>0$, as long as similar bounds are satisfied initially.
To see this, we first observe that 
after a simple reorganization of the various terms, Equation \eqref{eq:main1} can be rewritten as
\begin{equation}\label{eq:main2}
\pa_t u = \Big[ g(u) (1-K*p)  +K*( g(u)p)\Big] (1-p), \qquad p=P(u).
\end{equation}
The assumption $P(1)=1$ thus implies that the right-hand side of \eqref{eq:main2} vanishes when $u=1$. It follows that for any initial configuration such that $u(0,x)\leq 1$, the solution of \eqref{eq:main1} satisfies 
$$ u(t,x)\leq 1  \mbox{ for all } t>0.$$
By contrast, we recall that solutions of some non-local reaction diffusion \hb{equations} might overshoot and take values greater than $1$, which in turn leads to more complex behaviors such as oscillations near the front (compare \hb{e.g.}~\cite{BNPR}).
However, in our framework, the bound $u(t,x)\leq 1$ is warranted  as it accounts for the natural incompressibility of cells. 

\medskip

Next, we note that the normalization \eqref{eq:Knorm} implies that $K*p\leq 1$  so that the right-hand side of \eqref{eq:main2} is non-negative (recall that $g\geq 0$):
 any solution with $u(0,x)\in [0,1]$ will be monotone increasing in time and satisfy $u(t,x)\in [0,1]$ for all $t>0$. 
 This monotonicity is a fundamental feature of \eqref{eq:main1}. It is an important aspect of the evolution of a population of {\it immotile} cells, which  is not captured by classical reaction-diffusion equations such as the F-KPP equation for which solutions might, at least initially, be decreasing in time at some locations.

\medskip

\subsection{Singular limit}
Next, we derive a more explicit equation by considering the extreme case when movement only emanates from locations which are  saturated, that is where $u=1$, while cells that are created  at a location where $u<1$ stay there.
Such a regime appears in the singular {\it stiff}  limit
$\gamma\to\infty$ of the power-law relation $P(u)=u^\gamma$.
This limit leads to the condition that $p=0$ whenever $u<1$, but the condition \hb{$P(1)=1$} is lost in general, leading to the pressure law (sometimes referred to as the Hele-Shaw graph condition):
\begin{equation}\label{eq:HSgraph}
p  \in \begin{cases}
0 & \mbox{ when } u<1 \\
[0,\infty) & \mbox{ when } u=1\\
\infty & \mbox{ when } u>1.
\end{cases}
\end{equation}
We will show here that with the pressure law \eqref{eq:HSgraph}, the general model \eqref{eq:main1} reduces to the singular equation  \eqref{eq:main0}.
In Section \ref{sec:limit}, we will prove the convergence of \eqref{eq:main1} with $P(u)=u^\gamma$ to \eqref{eq:main0} in the limit $\gamma\to\infty$.
Note that this stiff limit  has been extensively investigated across various mathematical frameworks. A vast literature is notably devoted to  tumor growth models of the type (e.g. \cite{PQV,MPQ,DP,Perthame2014}:
$$ \pa_t u - \div( u \na p) = u f(p), \qquad p = u^\gamma.$$
It is well established that the limit $\gamma\to\infty$ of this
porous medium type equation leads to a free boundary problem of Hele-Shaw type.
We further discuss this in Section \ref{sec:HS}. Equation \eqref{eq:main0} can also be regarded as a free boundary problem but not of Hele-Shaw type. Instead, it has the structure of an obstacle problem. Indeed, we can rewrite it as: 
\begin{equation}\label{eq:obstacle}
\max\{u-1, \pa_t u -   [ g(u)  (1-  K*{\mathbb 1}_{\{u=1\}}) + g(1) K*{\mathbb 1}_{\{u=1\}} ]    \} = 0.
\end{equation}
This model provides a more geometrical description of the evolution of the saturated set $S(t)=\{u(t,\cdot)=1\}$ and the free boundary $\pa S(t)$.  We refer to our work with G. Steinbach~\cite{BMS2} for a more detailed discussion of the relevance of \eqref{eq:main0} to the description of range expansion of immotile cells.

\medskip

To obtain \eqref{eq:main0} from \eqref{eq:main1}, \eqref{eq:HSgraph}, we first observe that the condition 
``$p=0$ when $u<1$''
implies in particular that $g(u(t,x))p(t,x)=g(1)p(t,x)$ so that    using the form \eqref{eq:main2}, equation~\eqref{eq:main1} reduces to:
\begin{align}
\pa_t u &= \Big[ g(u) (1-K*p)  +K*( g(1)p)\Big] (1-p)\nonumber \\
&= \Big[ g(u) (1-K*p)  +g(1) K*p\Big] (1-p). \label{eq:fp}
\end{align}
We can now use the fact that $\pa_t u=0$ a.e. in $\{u=1\}$ (since $u\leq 1$ and $\pa_t u \geq 0$), to infer from \eqref{eq:fp} that 
$$ 0 = g(1) (1-p) \mbox{ in } \{u=1\}$$
and thus $p = {\mathbb 1}_{\{u=1\}}$.
Equation \eqref{eq:fp} then becomes
\begin{equation}\label{eq:main}
\pa_t u 
 = \Big[g(u)(1-K* \mathbb{1}_{\{u=1\}}) + g(1)K* \mathbb{1}_{\{u=1\}}\Big]\mathbb{1}_{\{u<1\}}
\end{equation}
which corresponds to \eqref{eq:main0}.

\medskip

The remainder of this paper is chiefly concerned with the study of \eqref{eq:main0}.

\section{Main results for the singular pressure model}\label{sec:results}
In the following we will first justify rigorously the singular limit leading to  \eqref{eq:main0} 
(from \eqref{eq:main1}).
Then, we will establish the fundamental properties of this model regarding the existence of traveling wave solutions and  spreading speed of propagation.

\subsection{Assumptions and definition}
Throughout the paper, we assume that $K$ is a probability density and a function of bounded variation,
\begin{equation}\label{eq:K1}
K\geq 0, \quad \int_{\R^d} K(x)\, dx =1, \quad \hb{TV(K)} = |\na K|(\R^d) <+\infty.
\end{equation}

While these assumptions are sufficient to derive \eqref{eq:main0} and prove its well-posedness, for
the long time asymptotic properties we will require 
the additional assumptions that $K$ be bounded, radially symmetric and compactly supported:
\begin{equation}\label{eq:K2}
K(x)= \hb{k}(|x|) \in L^\infty(\R^d), \quad \supp K  = B_\ell(0).
\end{equation}
The parameter $\ell>0$ can be thought of as representing the maximum distance that shuffling can reach at congestion.

Concerning the growth rate $g$, we assume
\begin{equation}\label{eq:f1}
 g(0)=0, \quad g(u)\geq  r u \quad  \forall u\in(0,1] \mbox{ for some $r>0$}, \quad \hb{g \mbox{  Lipschitz on } [0,1] \mbox{ with } L:=\mathrm{Lip} g}.
\end{equation}
\hb{Note that, since $g(0)=0$, this implies $\sup_{[0,1]} g \leq L$.}
A simple example is $g(u)=r_0 u$, but we can also consider nonlinear functions $g$ that take into account the competition for resources (such as nutrients), with $g(u) = r_0u(1-\frac{u}{M})$ for some carrying capacity $M\geq1$.
The sign of $g(1)-g(u)$ plays an important role in the equation, and some properties of the model 
(in particular the comparison principle) require the additional assumption
\begin{equation}\label{eq:f2}
g(u)\leq g(1) \mbox{ for all } u\in [0,1].
\end{equation}
This condition is typically satisfied when $g$ is monotone increasing on $[0,1]$ (if $g(u) = r_0u(1-\frac{u}{M})$, it holds as long as $M>2$).

\medskip

Despite the apparent simplicity of \eqref{eq:main0}, its well-posedness 
is not immediate. 
The non locality in $x$ and the singular definition of the pressure $p={\mathbb 1}_{u=1}$ raise some significant challenges. 
For our analysis, it will be convenient to view   \eqref{eq:main0} as an infinite  system of coupled ODEs parametrized by $x\in \R^d$. We will thus adopt the following definition of solutions:
\begin{definition}\label{def:sol}
The function $u(t,x)$ is a solution of \eqref{eq:main0} if for all $t>0$, $u(t,\cdot)$ is a function in $L^\infty (\R^d)$ satisfying $0\leq u(t,\cdot)\leq 1$ and if for any $x\in \R^d$, the function $t\mapsto u(t,x)$ is  a Lipschitz  function which solves \eqref{eq:main0} a.e $t>0$.

\hb{Similarly, the function $v(t,x)$ is a subsolution of \eqref{eq:main0} if for all $t>0$, $v(t,\cdot)$ is a function in $L^\infty(\R^d)$ satisfying $0\leq v(t,\cdot)\leq 1$, and if for any $x\in\R^d$ the function $t\mapsto v(t,x)$ is Lipschitz and satisfies
$$
\pa_t v \leq \big[ g(v) (1- K*{\mathbb 1}_{\{v=1\}}) + g(1) K*{\mathbb 1}_{\{v=1\}}\big] {\mathbb 1}_{\{v<1\}} \qquad \mbox{ a.e. } t>0 .
$$
Supersolutions are defined by the reverse inequality.}
\end{definition}

\medskip

\subsection{Main results}

\medskip

While solutions of 
\eqref{eq:main0} are naturally Lipschitz in time,  \eqref{eq:main0} has no obvious regularizing mechanisms with respect to the $x$ variable (unlike reaction-diffusion equations). 
However, we will show that the equation propagates some regularity of the initial data. 
For simplicity, we restrict ourselves to Lipschitz initial data so that solutions will be Lipschitz in both $t$ and $x$. 
The approach that we develop here could easily be extended to initial data that are only (locally)  H\"older continuous. 
Throughout the paper, we thus assume that the initial state $u^{in}$ is a Lipschitz function:
$$
\|u^{in}\|_{\Lip(\R^d)}: = \sup_{(x,y)\in \R^d, x\neq y} \frac{|u^{in}(x)-u^{in}(y)|}{|x-y|} <\infty.
$$

\noindent{\bf Singular limit.}
Our first task is the  rigorous derivation of \eqref{eq:main0} as the singular limit of \eqref{eq:main1}.
We thus consider the equation \eqref{eq:main1} (or equivalently \eqref{eq:main2})  with $P(u)=u^\gamma$, which we rewrite here for convenience:
\begin{equation}\label{eq:gamma}
\pa_t u = \Big[ g(u) (1-K*p)  +K*( g(u)p)\Big] (1-p), \qquad p=u^\gamma.
\end{equation}

The first result is:
\begin{theorem}\label{thm:limit}
Assume that $K$ and $g$   satisfy respectively \eqref{eq:K1} and \eqref{eq:f1}.
For all $\gamma\geq 1$, Equation \eqref{eq:gamma} has a unique solution $u_\gamma(t,x)$ satisfying $u_\gamma(0,x)= u^{in}(x)$.
When $\gamma\to+\infty$, $u_\gamma$ converges,  up to a subsequence, locally uniformly in $\R_+\times \R^d$ to a function $u_\infty(t,x)$ solution of \eqref{eq:main0} with $u_\infty(0,x)=u^{in}(x)$. It satisfies $0\leq u_\infty(t,x)\leq 1$ and $\pa_t u_\infty(t,x)\geq 0$ a.e. in $\R_+\times\R^d$.
Furthermore, for all $T>0$ we have
$$u_\infty\in \Lip ([0,T]\times\R^d), \quad {\mathbb 1}_{\{u_\infty=1\}} \in  BV_{loc}([0,T]\times \R^d).$$ 
\end{theorem}
This theorem, proved in Section \ref{sec:limit}, provides in particular the existence of a solution to \eqref{eq:main0}. 
As noted above, we do not need to work with Lipschitz initial data to get such a result: H\"older continuous initial data will yield H\"older continuous solutions.
On the other hand, Lipschitz regularity is the highest regularity that we can propagate: even with smooth initial data, the solution of \eqref{eq:main0} will not be~$C^1$ in general. In fact, as we will see in section~\ref{sec:TW}, we can construct traveling wave solutions that are Lipschitz continuous but not $C^1$ (see Figure \ref{fig:tw}).
\medskip

\noindent{\bf Comparison principle and uniqueness.}
The uniqueness of the solution constructed in Theorem~\ref{thm:limit} (which would guarantee the convergence of the whole sequence $\{u_\gamma\}$ in Theorem \ref{thm:limit}) is not obvious: The singular definition of the pressure ${\mathbb 1}_{\{u=1\}}$ together with the non-locality  of the convolution with $K$ 
do not seem to accommodate classical arguments.
However, we will prove in Section \ref{sec:comparison} that  \eqref{eq:main0}  has a comparison principle when $g$ also satisfies \eqref{eq:f2}, from which uniqueness will follows:
\begin{proposition}[Comparison principle]\label{prop:comp} 
    Assume that $g$ satisfies \eqref{eq:f1}-\eqref{eq:f2} and that $K$ satisfies \eqref{eq:K1}.
   Let $u_1(t,x)$ be a solution \hb{and $u_2(t,x)$ a subsolution} of \eqref{eq:main0} with Lipschitz initial data $u_1^{in}  $ and $u_2^{in} $ satisfying
    $$ 0\leq u_2^{in}(x)\leq u_1^{in}(x) \leq 1\quad \forall x\in \R^d.$$
    Then
    $$ 0\leq u_2(t,x) \leq u_1(t,x) \leq 1 \quad \forall x\in \R^d, \; \forall t>0.$$
\end{proposition}
It is interesting to note that this comparison principle does not hold without the condition \eqref{eq:f2} on $g$.
Indeed, if 
there exists $\alpha\in(0,1)$ such that $g(\alpha) > g(1)$, 
we can construct a simple counterexample in dimension 1 as follows (here we are also assuming that  $K>0$ in
$(-\hb{\ell}, \hb{\ell})$):
Define $u_2^{in} \equiv \alpha$ and let $u_1^{in}$ be a smooth function such that $u_1^{in}\geq u_2^{in} $ in $\R$, $u_1^{in} (x)=1 $ for $x<0$ and $u_1^{in} (\ell/2) = \alpha$.
Then $\pa_t u_2(0,\ell/2) = g(\alpha)$ while  
$\pa_t u_1(0,\ell/2) = g(\alpha) (1-K*{\mathbb 1}_{x<0}) + g(1) K*{\mathbb 1}_{x<0} < g(\alpha)$ (since $K*{\mathbb 1}_{x<0}(\ell/2)>0$).
It follows that, at least for short time, $u_1(t,\ell/2) < u_2(t,\ell/2)$.

\medskip

The comparison principle implies:
\begin{corollary}
Under the conditions of Proposition \ref{prop:comp},
let $u^{in}(x)\in [0,1]$ be a Lipschitz function in $\R^d$. Equation \eqref{eq:main0} has a unique solution satisfying  $u(0,x)=u^{in}(x)$.
\end{corollary}

\medskip

\noindent{\bf Invasion.}
If we further assume that $K>0$ in $B_\ell(0)$, we can show that 
the colony will invade the whole space. 
This invasion property  follows from the time monotonicity  of the solutions and the fact that \hb{a stationary solution of \eqref{eq:main0} either vanishes almost everywhere or is identically equal to $1$} (see Lemma \ref{lem:stat}):
\begin{proposition}\label{prop:invasion}
Assume that $K$ and $g$   satisfy respectively \eqref{eq:K1} and \eqref{eq:f1} and that $K>0 $ in $B_\ell(0)$.
Any solution $u(t,x)$ of \eqref{eq:main0} such that $u(0,\cdot)$ is not identically zero satisfies 
$$u(t,x)\nearrow 1 \quad \mbox{ as  } t\nearrow \infty, \; \mbox{ for all } x\in\R^d.$$
\end{proposition}
Furthermore, we will prove that this invasion occurs at least at linear speed, that is:
\begin{lemma}\label{lem:invasion}
\hb{Assume that $K$ and $g$ satisfy respectively \eqref{eq:K1} and \eqref{eq:f1}, and that $\Supp K = B_\ell(0)$. If $u^{in}(x)>\eta$ in $B_\delta(0)$} for some positive $\eta$ and $\delta$, there exist $\underline c>0$ and $t_0\geq 0$ such that: 
\begin{equation}\label{eq:ndgsp}
 \{u(t,\cdot)=1\} \supset B_{\underline c(t-t_0)_+}(0).  
\end{equation}
\end{lemma}

Next, we will characterize more precisely 
the expansion of the saturated set (and the support) of the solution as $t\to\infty$.

\medskip

\medskip

\noindent{\bf Traveling waves.}
One of the great successes of reaction-diffusion equations lies in
their ability to precisely describe the spreading of initially compactly supported solutions (speed and profile) by means of traveling wave solutions. 
The aim of the second part of this paper is to establish similar properties for \eqref{eq:main0}.  
We will prove the existence of planar traveling waves with speed $c\in [c^*,+\infty)$ for some $c^*>0$, and that compactly supported initial data spread with speed $c^*$.

We recall that planar traveling waves are solutions of the equation of the form $u(t,x) = \phi(x\cdot e-ct)$ with $e\in\mathbb S^{d-1}$ (the direction of propagation) and $c>0$ (the speed of propagation) and where the {\em profile}  $\phi : \R \to [0,1]$ connects the two stationary states:
$$ \lim_{s\to -\infty} \phi(s) = 1, \qquad  \lim_{s\to +\infty} \phi(s) = 0. $$
The existence of such solutions is a fundamental feature of reaction-diffusion equations and  the following theorem should be compared to the classical results for the \hb{Fisher}-KPP equation (\cite{KPP1937, AW}):
\begin{theorem}\label{thm:TW}
Assume that $g$ satisfies \eqref{eq:f1}-\eqref{eq:f2} and that $K$ satisfies \eqref{eq:K1}-\eqref{eq:K2}.
There exists a minimal speed $c^*>0$ such that for all $c\geq c^*$, \eqref{eq:main0} has a  unique (up to translation) traveling wave solution $\phi_c$ with speed $c$ while no traveling wave exists with speed $c<c^*$. 

For all $c\geq c^*$, the saturated set of the profile $\phi_c$ is  a half-line, so that   
up to translation, it satisfies  $\{\phi_c=1\}=(-\infty,0]$. Then, the followings hold:
\begin{itemize}
\item[(1)]    $\phi_c(s)>0$ for all $s\in \R$, for all $c>c^*$,
\item[(2)] \hb{$\{\phi_{c^*}>0\} = (-\infty,\ell)$.}
\end{itemize}
\end{theorem}
This theorem will be proved in Section \ref{sec:TW}. 

\medskip
\noindent{\bf Spreading.} Next we discuss spreading properties. The existence of traveling wave solutions and the comparison principle imply a bound on the speed of spreading. In our setting, since $\phi_{c^*}$ is semi-compactly supported, we actually get a bound on the speed of propagation of the support of the solution:
\begin{corollary}\label{cor:supp}
Under the assumptions of Theorem \ref{thm:TW},
    let $u^{in}:\R^d\to [0,1]$ be a Lipschitz  initial state  with $\supp u^{in} \subset B_R(0).$
    Then the corresponding solution  $u(t,x)$ of \eqref{eq:main0} satisfies
    $$ \Supp(u(t,\cdot)) \subset B_{R+\ell+c^*t}(0).$$
\end{corollary}

Interestingly, for compactly supported initial conditions,
the support of $u(t,\cdot)$ and the saturated set 
$$S(t)= \{x\in \R^d\, ;\, u(t,x)=1\}$$
remain close for all $t$:
The following result (proved in Section \ref{sec:spreading}) shows that the support of the solution is confined to a finite neighborhood of the congested  set:
\begin{proposition}\label{prop:support}
Under the assumptions of Theorem \ref{thm:limit}, assume further that $\Supp \, K \subset B_\ell$.
Then any solution  $u(t,x)$ of \eqref{eq:main0} satisfies
$$\mathrm{Supp} \, u(t,\cdot)  \subset \mathrm{Supp} \, u(0,\cdot) \cup \big(  S(t) + B_\ell\big) , \qquad    
\mbox{for all $t>0$.}$$
\end{proposition}
In particular for $t$ large enough so that $\mathrm{Supp} \, u(0,\cdot) \subset S(t)$ (that is once the saturated set overtakes the initial support of $u$), then 
$$ 
S(t) \subset \mathrm{Supp} \, u(t,\cdot)  \subset  S(t) + B_\ell.
$$
The region $\{0<u(t,\cdot)<1\}$ is thus contained in an annulus of width $\ell$ around the saturated set $S(t)$. These results will be proved in Section \ref{sec:spreading}).

\medskip

\noindent {\bf Asymptotic speed of spreading.}
We now show that $c^*$  exactly characterizes the speed of propagation starting from any compactly supported initial configuration.
\begin{theorem}\label{thm:asymptotic-spreading}
Assume that $g$ satisfies \eqref{eq:f1}-\eqref{eq:f2}, and that $K$ satisfies
\eqref{eq:K1}-\eqref{eq:K2}. 
Let $u^{\rm in}\colon\mathbb R^d\to[0,1]$ be Lipschitz continuous,
compactly supported, and not identically zero.  Then the corresponding
solution $u$ of \eqref{eq:main0} propagates with speed $c^*$. More precisely,
\begin{equation}
    \lim_{t\to\infty} \{\sup_{|x| \geq ct} u(t,x) \}=0
        \qtext{ for every }c>c^*, \label{eq:spreading-upper}
\end{equation}
and 
for every $c<c^*$,  there exists $T_c>0$ such that
\begin{equation}\label{eq:ballinvasion}
    B_{ct}(0)\subset\{u(t,\cdot)=1\}
    \qtext{for every }t\ge T_c.
\end{equation}

\end{theorem}

The lower bound for the spreading speed given in \eqref{eq:ballinvasion} is a stronger statement than for the classical KPP reaction-diffusion case where it takes the form
\begin{equation}
    \lim_{t\to\infty}\{ \sup_{|x|\leq ct} |1- u(t,x)|\}=0
        \qtext{for every }c\in[0,c^*). \label{eq:spreading-lower}
\end{equation}

This theorem shows that $c^*$ is the exact spreading speed for solutions having a compactly supported initial condition.
The proof is given in Section \ref{sec:spreading}.
It relies on the construction of a radially symmetric subsolution that propagates with speed close to $c^*$.

\section{A brief review of previous models}\label{sec:compliterature}

With a view to highlight some differences of our approach with respect to previous frameworks, we briefly review here various classes of models for spreading phenomena in biological systems involving reaction-diffusion mechanisms. The aim is to describe broadly their motivations, characteristic features and how the equation we introduce here relates to them. 
Within this perspective, this review lays no claim to being exhaustive.

\subsection{Linear diffusion: Fisher-KPP}
 
Reaction-diffusion equations, most notably the classical \hb{Fisher}-KPP equation \cite{KPP1937,Fisher1937}
\begin{equation}\label{eq:KPP}
\pa_t u - \Delta u = u(1-u),
\end{equation}
have played a central role in the mathematical analysis of invasion phenomena.
The diffusion operator is typically associated to random movement of individuals and the right hand side describes growth inhibited by the competition for resources: The threshold $u=1$  is interpreted as the environment {\it carrying capacity}.

First introduced for the study of geographical spread of genetic traits, this model is a cornerstone of the mathematical study of biological invasions. However, it exhibits several limitations in our setting. In particular, the infinite propagation speed implied by linear diffusion, together with the persistence of spreading even when the reaction term vanishes, indicates that the model captures active motion mechanisms and is thus not well suited for immotile cells which spread as a result of growth.
We also note that solutions of \eqref{eq:KPP} are not expected to be monotone in time, at least initially.

\subsection{Non-linear diffusion}
In dense populations, expansion is primarily driven by competition for space. Darcy's law, $v = -\nabla p$, describes the tendency of individuals to move away from congested regions. When coupled with a $\gamma$-pressure law, this leads to nonlinear degenerate diffusion models of the form
\begin{equation}\label{eq:KPPnl}
\pa_t u - \div (u \na p) = u(1-p), \quad p=u^\gamma.
\end{equation}
The use of such porous medium-type models in biological population dynamics was proposed for example by Gurtin and MacCamy \cite{GurtinMacCamy1977}. We also refer to Sherratt and Murray \cite{SherrattMurray1990} for applications to wound healing, and to Byrne and Drasdo \cite{byrne09} for more recent models of tumor growth.

The pressure $p$ is introduced here on phenomenological grounds to capture the effects of congestion. Notably, proliferation on the right-hand side is now regulated by the pressure: the threshold $p=1$, referred to as the {\it homeostatic pressure}, replaces the notion of carrying capacity, which is typically associated with competition for resources rather than space. 
Equation \eqref{eq:KPPnl} propagates the support of solution with finite speed and is thus better suited to the description of cell population dynamics.
However, as in \eqref{eq:KPP}, invasion can still occur in the absence of growth, and the lack of monotonicity in time remains an issue.

\subsection{Singular pressure and Hele-Shaw free boundary problems}\label{sec:HS}
There is a long history of free boundary problems being used to model cell culture and tumor growth, going back to 
Greenspan \cite{Greenspan,Greenspan1976} (see also Byrne and Chaplain \cite{Byrne1996,ByrneChaplain1997} and references therein).
The  mathematical analysis for such models, which typically account for many other phenomenon  (e.g. the role of surface tension and nutrients), was developed in particular by A. Friedman and his collaborators (see survey \cite{Friedman2015} and references therein).

More recently, there has been a lot of interest for the 
incompressible limit $\gamma\to\infty$ of the nonlinear diffusion model \eqref{eq:KPPnl} following the seminal work of Perthame, Quir\'os and V\'azquez \cite{PQV} - see also for example \cite{MPQ,DP,Perthame2014,Guillen2022}. 
This singular limit leads to a free boundary problem of Hele-Shaw type:
\begin{equation}\label{eq:HS}
\hb{-\Delta p =1-p} \mbox{ in } \{u=1\}, \qquad V = |\na p| \mbox{ on } \pa\{u=1\}
\end{equation}
where $V$ denotes the normal speed of propagation of the set $\{u=1\}$.
Such incompressible limits of the porous medium equation were first studied without a reaction term (see, for instance, \cite{Benilan1989,Gil2001}). A related approach to congestion also appears in models of crowd motion \cite{S_survey,Maury11,M18,MRS,CKY,MRSV,Alexander2014}, where congestion arises from advection rather than growth.

The free boundary problem \eqref{eq:HS} provides a simple geometric description of proliferation-driven growth in incompressible populations. 
It yields solutions that are monotone increasing in time (see \cite{PQV});  in the absence of a growth term, these limiting solutions are stationary, indicating that spatial expansion is a result of growth rather than active motion of the cells.

Nevertheless, these models do not capture all the features of interest. In particular, they do not account for the partitioning of the colony into an inner, reproduction-inhibited (congested and essentially frozen) region and an outer active rim where proliferation and mechanical pushing occur. Yet, this partitioning is usually observed in immotile cell colonies (compare e.g. \cite{BMS2} and the references therein).

\medskip

\subsection{Nonlocal dispersal}
The models discussed above are all local in nature. 
There is also a vast literature that represents dispersal phenomena with the aid of non local integral operators. But they all differ from the models \eqref{eq:main1} and \eqref{eq:main0} in significant ways. 

First, the most natural nonlocal version of a reaction-diffusion equation such as the Fisher-KPP equation \eqref{eq:KPP} is the following non-local dispersion equation:
\begin{equation}\label{eq:nlD}
\pa_t u = f(u) + \int_{\R^d} K(y-x) [u(y) - u(x)]\, dy
\end{equation}
where $f$ denotes the growth term.
It has been studied for instance in \cite{Coville05,BC16,Coville2007}. Many such equations are reviewed in the book \cite{roquejoffre2024frontpropagation}.
 
Our nonlinear integral operator $\L[u]$ describes a process that is very different from the linear dispersion of \eqref{eq:nlD}. In fact, equation~\eqref{eq:nlD} like the Fisher-KPP equation, separate the linear dispersal part and the growth terms. Whereas
here, $\L[u]$  involves  $g(u)$ (rather than $u$). Thus, in the framework of \eqref{eq:main0}   dispersion hinges on growth. Its dependence on the pressure $p$ further highlights the role of competition for space in the dispersion mechanism.

Other models with nonlocal dispersion have been 
used in particular for biological invasions in spatial ecology (Skellam \cite{Skellam1951}, Kot, Lewis and van den Driessche \cite{KotLewisVDD1996}) and epidemiology (Kendall \cite{K}).
This has motivated the mathematical studies of the prototype equation~\cite{Neuhauser2001,PS05}
$$
\pa_t u = -u + \lambda (1-u) K*u.
$$
We also mention the neural field equation \cite{EM93,TLS25} which takes the form:
$$
\pa_t u = -u + K*S(u).
$$
Even though these equations bear some similarities with the model we discuss here, there are also notable differences. For example, these non-local equations, like the local equation~\hb{\eqref{eq:KPP}}, do not lead in general to monotone in time solutions. Note also that congestion and saturation are not taken into account in these models.

\subsection{Nonlocal competition term}
Finally, equations with local dispersion but non-local  production terms have also been used extensively in models of ecological invasion. A typical example is the  non-local \hb{Fisher}-KPP equation \cite{Furter,GVA,PG07}:
\begin{equation}\label{eq:nFKPP}
u_t = u (1-K*u ) + \Delta u
\end{equation} 
studied for instance in \cite{Britton89,Britton,Gourley,BNPR,HR}.

An important and well-documented feature of such equations is that their solutions do not, in general, satisfy the upper bound $u\leq 1$ and may exhibit overshoot. This phenomenon, in turn, can lead (at least in certain regimes) to oscillatory behavior of traveling wave solutions near the front (see \cite{BNPR}).

\smallskip
\subsection{\hb{Comparison with our model}}
Our model combines the Hele-Shaw approximation ($p=0$ when $u<1$) with  nonlocal dispersion phenomena. Like the Hele-Shaw approximation, it leads to monotone-in-time solutions, but it also includes a mesoscopic parameter $\ell$ which accounts for jamming phenomena and plays a key role in the partitioning of the colony proved in Proposition \ref{prop:support}. Like the more classical reaction-diffusion equation, it has traveling wave solutions which characterize the spreading behavior of any cell colony with  compactly supported initial support.

Lastly, the model we propose here presents a notable advantage of simplicity for numerical simulation with respect to most of the models mentioned above.

\bigskip

\section{Singular limit: Proof of Theorem \ref{thm:limit}}\label{sec:limit}

In this section we analyze the passage to the limit $\gamma\to\infty$ in \eqref{eq:gamma} in order to construct a 
solution of \eqref{eq:main0}.

We start by observing that since $u\mapsto P_\gamma(u):=u^\gamma$ is a Lipschitz function (for $\gamma\geq 1$), equation~\eqref{eq:gamma} has a unique solution for all initial data $u^{in}$.
\begin{proposition}\label{prop:existence}
Assume that $K$ and $g$   satisfy respectively \eqref{eq:K1} and \eqref{eq:f1} and let $\gamma\geq 1$.
Let $u^{in} (x)$ be a continuous function such that $0\leq u^{in}(x)\leq 1$ for all $x\in \R^d$. 
There exists a unique solution $u_\gamma(t,x)$  of \eqref{eq:gamma} with $u_\gamma(0,x)=u^{in}(x)$.
This solution satisfies $0\leq u_\gamma(t,x)\leq 1$ for all $x\in \R^d$ and $t>0$ and $t\mapsto u_\gamma(t,x)$ is a non-decreasing function for all $x\in \R^d$.
\end{proposition}
\begin{proof}
The result follows from Cauchy-Lipschitz-Picard's Theorem. We extend the function $g$  to $\R$ by constants outside $[0,1]$, that is, we set $g(u)=0$ for $u\leq 0$, $g(u)=g(1)$ for $u\geq 1$.  We also modify accordingly the definition of $P_\gamma$ so that $P_\gamma(u)=0$ for $u\leq 0$ and $P_\gamma(u)=1$ for $u\geq 1$.
We can then define the operator $T:C^0_b(\R^d) \to C^0_b(\R^d) $ by 
$$T[u](x) = 
 \Big[ g(u(x)) (1-K*P_\gamma(u)(x))  +K*( g(u)P_\gamma(u))(x)\Big] (1-P_\gamma(u)(x)).$$
It is easy to check that $T$ is a Lipschitz operator on $C^0_b(\R^d)$ (equipped with the sup norm) 
and the existence of a unique, continuously differentiable, solution of~\eqref{eq:gamma} $u:[0,\infty)\to C^0_b(\R^d)$ (with $u(t=0)=u^{in}$) follows.

\medskip

Since $0\leq P_\gamma(u)\leq 1$, we know that $0\leq K\ast  P_\gamma(u)\leq 1$ and so 
\begin{equation}\label{eq:rhspos}
0\leq g(u(x)) (1-K*P_\gamma(u)(x))  +K*( g(u)P_\gamma(u))(x) \leq \sup g\leq L\end{equation}
where $L$ is the Lipschitz norm of $g$ (from assumption~\eqref{eq:f1}).
Thus we have:
$$ 0 \leq \pa_t u \leq \gamma L \,  (1-u).$$
These inequalities, together with the assumption $0\leq u^{in}(x) \leq 1$ yield the bounds $0\leq u(t,x)\leq 1$  and  the monotonicity in time.
\end{proof}

In order to pass to the limit $\gamma\to\infty$, we now derive some estimates on the solution $u_\gamma$, the pressure $p_\gamma:=P_\gamma(u_\gamma)$ and their derivatives independent of $\gamma$.
More precisely, we are going to show that $u_\gamma$ is locally  Lipschitz (in $x$ and $t$) and that $p_\gamma$ is locally bounded in $BV$, the space of functions of bounded variation.
First we prove:
\begin{lemma}[A priori estimate for the density]\label{lem:compact}
Assume that $g$ satisfies \eqref{eq:f1} and that $K$ satisfy \eqref{eq:K1}.
Let $u^{in}$ be a Lipschitz function satisfying $0\leq u^{in}\leq1$.
The solution $u_\gamma$ of \hb{\eqref{eq:gamma}} given by Proposition \ref{prop:existence} satisfies
\begin{equation}\label{eq:tbd} 
|\pa_t u_\gamma(t,x)|\leq   L \qquad \mbox{ a.e. $t>0$ and for all $x\in \R^d$}.
\end{equation}
Furthermore, if $ u^{in} \in \Lip(\R^d)$, then 
\begin{equation}\label{eq:Lpbound}
    \sup_{t\in [0,T]} \|u_\gamma(t,\cdot)\|_{\Lip(\R^d)} \leq  \left ( \| u^{in} \|_{\Lip(\R^d)}  + 2 \hb{TV(K)} \right) e^{L T}.
\end{equation}
\end{lemma}
Note that the bound \eqref{eq:Lpbound} is not uniform in  $T$ and may suggest that the Lipschitz regularity of  $u_\gamma$
gets worse as time increases. This bound is not optimal. In fact we expect that even discontinuous initial data give rise to solutions that become uniformly Lipschitz in space after an initial boundary layer in time. 
\begin{proof}
The bound \eqref{eq:tbd} follows immediately from   \eqref{eq:rhspos}.

In the following, we drop the dependence of $u$ and $p$ on $\gamma$ to keep notations simple. 
For $x,y\in \R^d$, given we write:
\begin{align*}
\pa_t[ u(t,x)-u (t,y) ] & =
[g(u (t,x))-g(u (t,y)) ] (1-K\ast p(t,x))(1-p (t,x)) \\
& \quad  +[K \ast (g(u)p) (t,x) - K \ast (g(u)p) (t,y) ] (1-p (t,x)) \\
& \quad -g(u (t,y)\hb{)}[K \ast p (t,x) - K \ast p (t,y)] (1-p (t,x)) \\
& \quad - [ g(u (t,y))(1- K\ast  p (t,y) ) + K\ast (g(u) p) (t,y)] (p (t,x)-p (t,y)) .
\end{align*}
Multiplying by $\sign(u(t,x)-u(t,y))$ and noticing that $\sign(p(t,x)-p(t,y))=\sign(u(t,x)-u(t,y))$, we get:
\begin{align*}
\pa_t| u(t,x)-u (t,y) | & \leq 
L |u (t,x)-u (t,y)| (1-K\ast p(t,x))(1-p (t,x)) \\
& \quad  +|K \ast (g(u)p) (t,x) - K \ast (g(u)p) (t,y) | (1-p (t,x)) \\
& \quad + g(u (t,y)\hb{)}|K \ast p (t,x) - K \ast p (t,y)| (1-p (t,x)) \\
& \quad - [ g(u (t,y))(1- K\ast  p (t,y) ) + K\ast (g(u) p) (t,y)] |p (t,x)-p (t,y)|.
\end{align*}
Using \eqref{eq:rhspos}, we infer
\begin{align*}
\pa_t| u(t,x)-u (t,y) | 
&  \leq L   |u (t,x)-u (t,y)| +  L |K \ast p (t,x) - K \ast p (t,y)| \\
& \qquad +|K \ast (g(u)p) (t,x) - K \ast (g(u)p) (t,y) | .
\end{align*}
The classical inequality\footnote{Indeed, for a smooth kernel $K$, we can write
\begin{align*}
  |  K \ast w (x+h) - K \ast w (x)| & = \left|\int_{\R^d} (K(x+h-z) -K(x-z) )w(z)\, dz \right| \\
    & = \left| \int_{\R^d} \int_0^1 \na K (x+s h-z)\cdot h w(z)\,ds\, dz\right|\\
    & \leq \|w\|_{L^\infty} |h|  \int_{\R^d} \int_0^1 |\na K (x+s h-z)|\,ds\, dz  \leq \|w\|_{L^\infty} |h|  \int_{\R^d}   |\na K (z)|  dz
\end{align*}
and   \eqref{eq:lipK} follows by an approximation argument.
}
\begin{equation}\label{eq:lipK}
|K \ast w(x) - K \ast w(y)|\leq \|w\|_{L^\infty} \hb{TV(K)} 
  |x-y| 
\end{equation}
then implies
$$
\pa_t| u(t,x)-u (t,y) | \leq L |u (t,x)-u (t,y)| + 2L \hb{TV(K)} |x-y|.
$$
Integrating this differential inequality we find (for all $t>0$)
\begin{align*}
| u_\gamma(t,x)-u_\gamma(t,y) | 
& \leq  \left ( |u^{in}(x)-u^{in}(y)|  + 2\hb{TV(K)}  |x-y| \right) \hb{e^{L t}}\\
& \leq  \left ( \|u^{in}\|_{\Lip(\R^d)}   +2 \hb{TV(K)}   \right) \hb{e^{L t}}|x-y|
\end{align*}
and the result follows.
\end{proof}

\begin{lemma}[A priori estimate for the pressure]\label{lem:pressurecompact}
Under the conditions of Lemma \ref{lem:compact},
the pressure $p_\gamma =P_\gamma(u_\gamma)$ satisfies
$$
\int_0^T\int_{B_R} |\pa_t p_\gamma | + |\na_x p_\gamma|\, dx \, dt \leq |B_R|+
 2 \left(  T + \frac{4d}{r} \right) e^{LT} |B_R|  \Big[\|u^{in}\|_{\Lip} +2 \hb{TV(K)} \Big] 
\qquad \mbox{for all $\gamma\geq 3$}
$$
for all $T>0$ and $B_R\subset \R^d$,
 where   $r$ is given by \eqref{eq:f1}.
\end{lemma}
We note that this bound in $BV$ is optimal since we expect the pressure $p_\gamma$ to converge to a characteristic function when $\gamma\to\infty$.
\begin{proof}[Proof of Lemma \ref{lem:pressurecompact}]
First, we use the monotonicity in time to derive a bound on the time derivative. Since $\pa_t p_\gamma = P'_\gamma(u_\gamma)\pa_t u_\gamma \geq0$, we can write
$$ \int_0^T\int_{B_R} |\pa_t p_\gamma |\, dx \, dt = 
\int_0^T\int_{B_R} \pa_t p_\gamma \, dx \, dt \leq \int_{B_R}  p_\gamma (T,x)\, dx \leq |B_R|.$$

By Lemma \ref{lem:compact}, we know that $u_\gamma$ is a Lipschitz function. Therefore, we can differentiate \eqref{eq:gamma}  with respect to one of the coordinates $x_i$  to get
\begin{align*}
\pa_t(\pa_{x_i} u_\gamma ) & =
g'(u_\gamma) \pa_{x_i} u_\gamma (1-K\ast p_\gamma)(1-p_\gamma)
+ \big[-g(u_\gamma)  \pa_{x_i}K \ast p_\gamma+\pa_{x_i}K \ast (g(u_\gamma)p_\gamma)   \big]  (1-p_\gamma) \\
& \quad - \big[ g(u_\gamma)(1- K\ast  p_\gamma ) + K\ast (g(u_\gamma)  p_\gamma)\big] \pa_{x_i} p_\gamma .
\end{align*}
Multiplying by $\sign(\pa_{x_i} u_\gamma)$ and noticing that $\pa_{x_i }u_\gamma$ and $\pa_{x_i }p_\gamma$ have the same sign,
we deduce:
\begin{align}
\pa_t|\pa_{x_i} u_\gamma | & =
g'(u_\gamma) |\pa_{x_i} u_\gamma | (1-K\ast p_\gamma)(1-p_\gamma) \nonumber \\
& \quad + \big[-g(u_\gamma)  \pa_{x_i}K \ast p_\gamma+\pa_{x_i}K \ast (g(u_\gamma)p_\gamma)   \big]   \sign(\pa_{x_i} u_\gamma)  (1-p_\gamma)\nonumber  \\
& \quad - \big[ g(u_\gamma)(1- K\ast  p_\gamma ) + K\ast  (g(u_\gamma) p_\gamma)\big] |\pa_{x_i} p_\gamma| \nonumber \\
& \leq L  \left(|\pa_{x_i} u_\gamma | + 2 \int |\pa_{x_i} K(x)| \right) - \big[ g(u_\gamma)(1- K\ast  p_\gamma ) + K\ast (g(u_\gamma) p_\gamma)\big] |\pa_{x_i} p_\gamma| .\label{eq:gjhkf}
\end{align}
In order to conclude we make use of the following technical result:
\begin{lemma}\label{lem:tech}
Assuming that $u\mapsto g(u)$ satisfies  \eqref{eq:f1} and that $P_\gamma(u)=u^\gamma$ with $\gamma\geq 3$, we have
\begin{equation}\label{eq:claim}
g(u(x))(1- K\ast  P_\gamma(u)(x) ) + K\ast (g(u) P_\gamma(u))(x) \geq \frac r 4 u(x) , \qquad \forall x\in \R^d
\end{equation}
for any function $u(x)\in[0,1]$. 
\end{lemma}
Postponing the proof of this lemma, we see that using 
Inequality \eqref{eq:claim} in \eqref{eq:gjhkf} yields:
$$
\pa_t|\pa_{x_i} u_\gamma |
+ \frac{r}{4} u_\gamma |\pa_{x_i} p_\gamma|
 \leq L (|\pa_{x_i} u_\gamma | + 2 \hb{TV(K)} ).
$$
This differential inequality implies that
\begin{align*}\frac{r}{4} \int_0^T \int_{B_R} u_\gamma (t,x) |\pa_{x_i} p_\gamma (t,x)|\, dx\, dt 
& \leq \int_{B_R}  [|\pa_{x_i} u^{in}(x)| +2 \hb{TV(K)} ] \, dx \, e^{LT}\\
& \leq |B_R| \Big[ \hb{\|u^{in}\|_{\Lip}} +2 \hb{TV(K)} \Big] \, e^{LT}.
\end{align*}
Since this holds for all $i=1,\dots,d$, we obtain the following bound:
$$
\frac{r}{4} \int_0^T \int_{B_R} u_\gamma (t,x) |\na_x p _\gamma(t,x)|\, dx\, dt \leq |B_R|  \Big[ \sqrt d \hb{\|u^{in}\|_{\Lip}} +2 d \hb{TV(K)} \Big] e^{LT}.
$$
Finally, using
$$|\na  P_\gamma(u)| \leq
\begin{cases}
\gamma \left(\frac1 2 \right)^{\gamma-1} |\na u |
& \mbox{ if } u \leq \frac 1 2\\ 
2 u |\na P_\gamma(u)|
& \mbox{ if } u \geq 1/2 
\end{cases}
$$
and  $\gamma \left(\frac1 2 \right)^{\gamma-1}\leq 2$ for $\gamma\geq 1$ we get: 
\begin{align*}
\int_0^T \int_{B_R}   |\na_x p_\gamma (t,x)|\, dx\, dt 
&\leq 2 \int_0^T \int_{B_R} u_\gamma (t,x) |\na_x p_\gamma (t,x)|\, dx\, dt + 
2 \int_0^T \int_{B_R}    |\na_x u_\gamma(t,x) |\, dx\, dt\\
& \leq \frac{8e^{LT}}{r} |B_R| d \Big[\hb{\|u^{in}\|_{\Lip}} + 2\hb{TV(K)} \Big] + 
2 T |B_R| \|\na_x  u_\gamma \|_{L^\infty ([0,T]\times \R^d)}
\end{align*}
and the result follows using \eqref{eq:Lpbound}.
\end{proof}

\begin{proof}[Proof of Lemma \ref{lem:tech}]
The fact that $P_\gamma(u)\in [0,1]$, together with assumption \eqref{eq:f1}, yield:
$$
g(u(x))(1- K\ast  P_\gamma(u)(x) ) + K\ast (g(u) P_\gamma(u))(x) \geq   
r [u(x) (1- K\ast  P_\gamma(u)(x) ) + K\ast (u P_\gamma(u))(x)].
$$
Next, we note that $p(x)\geq \frac 1 2$ implies $u(x) \geq \left(\frac{1}{2}\right)^{1/\gamma}$,
so that we can write:
\begin{align*}
u (1- K\ast p) + K\ast (u p)
& \geq 
u \big[1- K\ast( p \mathbb 1_{p\geq \frac 1 2} )-K\ast( p \mathbb 1_{p< \frac 1 2}) \big] + K\ast (u p \mathbb 1_{p\geq \frac 1 2})\\
& \geq 
u \big[1- K\ast( p \mathbb 1_{p\geq \frac 1 2} )\big]- u K\ast( p \mathbb 1_{p< \frac 1 2}) ) +  \left(\frac{1}{2}\right)^{1/\gamma}K\ast (p \mathbb 1_{p\geq \frac 1 2}) \\
& \geq 
u \big[1- K\ast( p \mathbb 1_{p\geq \frac 1 2} )\big]+  \left(\frac{1}{2}\right)^{1/\gamma}K\ast (p \mathbb 1_{p\geq \frac 1 2}) - \frac1 2 u 
\end{align*}
where the last line follows from the fact that $K\ast( p \mathbb 1_{p< \frac 1 2})\leq K\ast \frac 1 2 = \frac 1 2$.
Since $K\ast( p \mathbb 1_{p\geq \frac 1 2} )\in [0,1]$, the convex combination $u (1- K\ast( p \mathbb 1_{p\geq \frac 1 2} )) +  \left(\frac{1}{2}\right)^{1/\gamma}K\ast (p \mathbb 1_{p\geq \frac 1 2})$ is \hb{greater} than $\min\left(u,\left(\frac{1}{2}\right)^{1/\gamma}\right) $, hence
\begin{align*}
u (1- K\ast p ) + K\ast (u p)
 & \geq \min\left(u,\left(\frac{1}{2}\right)^{1/\gamma}\right)  - \frac1 2 u.
\end{align*}
In order to conclude, it remains to notice that for $s\in [0,1]$, we have $\min\left(s,\left(\frac{1}{2}\right)^{1/\gamma}\right) \geq \frac 3 4 s$ as long as $\left(\frac{1}{2}\right)^{1/\gamma}\geq \frac 3 4 $, which holds for $\gamma$ large enough (e.g. $\gamma\geq 3$).
\end{proof}

We now have everything we need to pass to the limit $\gamma\to\infty$ and prove the existence of a solution of \eqref{eq:main0}.

\begin{proof}[Proof of Theorem \ref{thm:limit}]
By Lemma \ref{lem:compact} and Lemma \ref{lem:pressurecompact} we can extract a sequence $\gamma_j\to\infty$ (we \hb{still} denote it by $\gamma$ for simplicity) along which $u_\gamma(t,x)$ converges to $u_\infty(t,x)$ locally uniformly
and $p_\gamma(t,x)$ converges to $p_\infty(t,x)$ strongly in $L^1_{loc}([0,\infty)\times\R^d) $ and almost everywhere.
We note that $u_\infty$ is Lipschitz in $x$ and $t$ and $p_\infty$ is in $\mathrm{BV}_{loc}(\R_+\times\R^d)$.

We can thus pass to the limit in the products $g(u_\gamma) p_\gamma$ in the equation \eqref{eq:gamma} to show that the limit $(u_\infty,p_\infty)$ solves
$$
\pa_t u_\infty  =\big[ g(u_\infty)(1-K\ast  p_\infty) +   K\ast (g(u_\infty) p_\infty)\big](1-p_\infty) \quad \mbox{ a.e. in } \R_+\times\R^d.
$$
Furthermore, we know that $p_\gamma(t,x)$ converges to $0$ pointwise in  $\{u_\infty(t,x)<1\}$ (in fact the convergence is locally uniform in that open set) and thus $p_\infty(t,x) = 0 $ a.e. in $\{u_\infty(t,x)<1\}$.
In particular we have $g(u_\infty) p_\infty = g(1)p_\infty$ a.e. in  $\R_+\times\R^d$, which yields
\begin{equation}\label{eq:limit0}
\pa_t u_\infty =\big[ g(u_\infty)(1-K\ast  p_\infty) +  g(1) K\ast  p_\infty\big](1-p_\infty)  \quad \mbox{ a.e. in } \R_+\times\R^d.
\end{equation}
Finally, using the fact that $\pa_t u_\infty = 0$ a.e. in $\{u_\infty=1\}$ (since $u_\infty \in W^{1,\infty}(0,T;L^\infty(\R^d))$), Equation \eqref{eq:limit0} yields
$$ 0  = g(1)(1-p_\infty) \quad \mbox{ a.e. in } \{u_\infty=1\}$$
hence $p_\infty ={\mathbb 1}_{\{u_\infty =1\}}$  a.e. in $\R_+\times\R^d$.
\medskip

It remains to show that $t\mapsto u_\infty(t,x)$ solves \eqref{eq:main0} for all $x\in \R^d$ (see Definition \ref{def:sol}).
We thus fix $x_0\in \R^d$. Since $K\in L^1(\R^d)$, we have $K\ast p_\infty(t,x_0)=K\ast {\mathbb 1}_{\{u_\infty =1\}} (t,x_0)$ a.e. $t>0$ and the following convergence holds:
$$K\ast p_\gamma(t,x_0)\to K\ast {\mathbb 1}_{\{u_\infty =1\}}(t,x_0) \quad\mbox{ a.e. } t>0. $$
It should be noted that this limit holds pointwise in $x_0$ and not just a.e.
Next, the monotonicity and continuity of $t\mapsto u_\infty(t,x_0)$  (also pointwise in $x_0$) implies the existence of $t_0(x_0)\in [0,\infty]$ such that $u_\infty(t,x_0)<1$ for $t<t_0(x_0)$ and $u_\infty(t,x_0)=1$ for $t\geq t_0(x_0)$.
 Since $p_\gamma(t,x_0)$ converges to $0$ (locally uniformly) for $t\in (0,t_0(x_0))$, we can pass to the limit in \eqref{eq:gamma} to get
$$
\pa_t u_\infty (t,x_0) =[ g(u_\infty) (1-K\ast  {\mathbb 1}_{\{u_\infty =1\}}) + g(1)  K\ast  {\mathbb 1}_{\{u_\infty =1\}} ](t,x_0) \qquad \mbox{ a.e. } t\in (0,t_0(x_0)).
$$
Finally, we obviously have $\pa_t u_\infty (t,x_0) = 0 $ for all $t>t_0(x_0)$, and so
$$
\pa_t u_\infty   =[ g(u_\infty) (1-K\ast  {\mathbb 1}_{\{u_\infty =1\}}) + g(1)  K\ast  {\mathbb 1}_{\{u_\infty =1\}}](1-{\mathbb 1}_{\{u_\infty =1\}} ) \qquad 
\mbox{ a.e. } t>0.
$$
This equality holds for every $x\in \R^d$, so $u_\infty(t,x)$ is a solution of \eqref{eq:main0} in the sense of Definition~\ref{def:sol}.
\end{proof}

\section{Comparison principle and uniqueness}\label{sec:comparison}
In the previous section, we showed that given a Lipschitz initial condition,
\eqref{eq:main0} has a solution that is Lipschitz in both $x$ and $t$.
In order to prove uniqueness, we first prove that any solution must be 
Lipschitz:
\begin{lemma}\label{lem:Lip}
    Assume that $g$ satisfies \eqref{eq:f1} and that $K$ satisfies \eqref{eq:K1}.
    Let $u(t,x)$ be a solution of \eqref{eq:main0} in the sense of Definition \ref{def:sol} with initial state $u^{in} \in \mathrm {Lip} (\R^d)$. Then $u \in \mathrm {Lip} ([0,T]\times \R^d)$ for all $T>0$. More precisely,  the following pointwise estimate holds:
$$
|u(t,x)-u(t,y)| \leq e^{Lt} \Big[ |u^{in}(x)-u^{in}(y)| + \hb{TV(K)}  |x-y|\Big].
$$
\end{lemma}
\begin{proof}
The proof follows closely that of Lemma \ref{lem:compact}.
As a preliminary remark, we note that given $u(t,x)$ a  solution of \eqref{eq:main0} and an arbitrary function $v(t,x)$ satisfying $v(t,x)\leq 1$ almost everywhere, we have:
$$ 
 \sign(u-v)\pa_t u \leq [g(u) (1- K*{\mathbb 1}_{\{u =1\}}) + g(1)  K*{\mathbb 1}_{\{u =1\}}]\sign (u-v).
$$
Indeed, this is an equality when $u<1$ and it follows from 
$\pa_t u=0$, $g\geq 0$, and
$\sign (u-v)\geq 0$ 
when $u=1$.

Given $h\in \R^d$, we set $v(t,x) = u(t,x+h)$. Since $u$ and $v$ both solve  \eqref{eq:main0} and satisfy $u,\;v\leq1$, we can write the two inequalities:
$$\sign(u-v) \pa_t u  \leq [g(u) (1- K*{\mathbb 1}_{\{u =1\}}) + g(1)  K*{\mathbb 1}_{\{u =1\}}]\sign(u-v),  $$
and 
$$\sign(v-u) \pa_t v  \leq [g(v) (1- K*{\mathbb 1}_{\{v =1\}}) + g(1)  K*{\mathbb 1}_{\{v =1\}}]\sign(v-u) .$$
Adding these inequalities, we deduce:
\begin{align*}
\sign(u-v)\pa_t (u-v)  & \leq  [(g(u) -g(v))(1- K*{\mathbb 1}_{\{v =1\}})\\
 & \qquad\qquad + (g(1)-g(u)) (K*{\mathbb 1}_{\{u =1\}}-K*{\mathbb 1}_{\{v =1\}}) ]\sign(u-v)
\end{align*}
which implies
\begin{equation}\label{eq:gron1}
  \pa_t |u-v| \leq   L |u-v| + L |K*{\mathbb 1}_{\{u =1\}}-K*{\mathbb 1}_{\{v =1\}}|.
\end{equation}

We now observe that $K*{\mathbb 1}_{\{v =1\}}(x)= \int K(x-y) {\mathbb 1}_{\{u(t,\cdot+h)=1\}}(y)\, dy = K*{\mathbb 1}_{\{u =1\}}(x+h)$
and so \eqref{eq:lipK} gives
\begin{align*}
|K*{\mathbb 1}_{\{u =1\}}-K*{\mathbb 1}_{\{v =1\}}| (x) 
& \leq  TV(K)|h|
\end{align*}
Going back to \eqref{eq:gron1}, we now find
$$
 \pa_t |u-v| \leq  L |u-v| + TV(K)  L   |h|.
$$
and we can conclude as in the  proof   of Lemma \ref{lem:compact}.
\end{proof}

Next, we prove a  non-degeneracy estimate on the linear growth of the saturated set as it provides a positive lower bound on speed of spreading of this set.
\begin{lemma}\label{lem:sat}
    Under the conditions of Lemma \ref{lem:Lip}, and given 
    $(t_0,x_0)$ such that $u(t_0,x_0)=1$,
    there exist constants $r_1,c_1\in(0,+\infty)$ depending on $t_0$, such that
    $$
    u(t,x) = 1 \qquad \mbox{ for all } |x-x_0|\leq r_1 \mbox{ and } t \geq t_0 + c_1 |x-x_0|.
    $$
\end{lemma}
\begin{proof}
    Since $u(t_0,x_0)=1$, the Lipschitz regularity (Lemma \ref{lem:Lip}) implies 
    \begin{equation}\label{eq:minux}
        u(t_0,x)\geq 1-C_0|x-x_0|
    \end{equation} 
    for some constant $C_0$ depending on $t_0$. In particular, $u(t_0,x)\geq \frac 1 2$ if $|x-x_0|\leq r_1 : = \frac{1}{2C_0}$.
    It remains to show that this implies that $u(t,x)$ reaches the value $1$ in a (truncated) cone $\{(t,x)\, ; \, |x-x_0|\leq r_1 \mbox{ and } t \geq t_0 + c_1 |x-x_0| \}$.
    Indeed, if  $u(t_0,x)<1$, then we denote by $t_*$ the first time when $u(t_*,x)=1$ (with $t_*=\infty$ if this never happens). Since we have \hb{$\pa_t u \geq \min\big(g(u),g(1)\big)$} for $t\in (t_0,t_*)$  we get
    $$
    \int_{u(t_0,x)}^1 \frac{1}{\hb{\min(g(s),g(1))}}\, ds \geq t_*-t_0 .
    $$
    Assumption \eqref{eq:f1} implies \hb{$\min(g(s),g(1))\geq \frac r 2$} for $s\in (1/2,1)$ and so this inequality gives
    $$ 
    t_*-t_0 \leq \frac 2 r (1-u(t_0,x)) \leq \frac {2C_0} r   |x-x_0| \qquad \mbox{ for all }|x-x_0|<r_1.
    $$
The result follows.    
\end{proof}

\medskip
 We now have all the ingredients needed to prove a comparison principle for \eqref{eq:main0}.

\begin{proof}[Proof of Proposition \ref{prop:comp}]
Given $u_1(t,x)$ a solution \hb{and $u_2(t,x)$ a subsolution} of \eqref{eq:main0} with Lipschitz initial data  satisfying $u_1^{in}(x)\geq u_2^{in}(x)$ for all $x\in \R^d$, we consider the function
$$ v_1(t,x) = u_1((1+\delta) t +\eta,x)$$
with small $\delta,\eta>0$, both meant to go to zero.
We are going to show that
$$
v_1(t,x) \geq u_2(t,x) \qquad\mbox{for all } t>0, \; x\in \R^d.
$$
The role of $\delta$ is to make $v_1$ a strict supersolution of the equation, while $\eta$  ensures that the saturated  sets $\{ v_1=1\}$ and $\{ u_2=1 \}$  of $v_1$ and $u_2$ are strictly separated at time $t=0$.  Lemma \ref{lem:sat} (with $t_0=0$) implies that 
\begin{equation}\label{eq:v1u1}
v_1(0,x) = u_1(\eta,x) =1 \mbox{ for all } x \in \{u_1^{in}=1\}+B_{\eta/{c_1}}   
\end{equation}
and the fact that $u_2^{in}(x)\leq u_1^{in}(x)\leq 1$
then implies
$$
\{u_2^{in}=1\}+B_{\eta/{c_1}} \subset \{v_1^{in}=1\}.
$$

The idea of the proof of the comparison principle is to 
look at the first time $t_*$ at which the free boundaries $\pa \{v_1(t) = 1\} $ and $\pa\{u_2(t)=1\}$ meet, which happens at some point $x_*$. For all $t<t_*$, the pressures ${\mathbb 1}_{\{v_1 =1\}}$ and ${\mathbb 1}_{\{u_2 =1\}}$ are then ordered. This will allow us to compare the nonlocal terms $K*{\mathbb 1}_{\{u =1\}}$ in the equations for $v_1 $ and $u_2$ and derive a contradiction from the equality $v_1(t_*,x_*)  = u_2(t_*,x_*)=1$, since $v_1$ is a strict supersolution of the equation.
\medskip

To make this argument rigorous, we introduce the open set $\Omega_1(t)  = \{ v_1(t,\cdot)<1\}$ and define
$$ \Lambda(t)  = \sup_{\overline {\Omega_1(t)}} u_2(t,x).$$
The strict separation property \eqref{eq:v1u1} implies that
$$\Lambda(0)<1.$$

Indeed, for all $x_0\in\Omega_1(0)$ we have $v_1(0,x_0)= u_1(\eta,x_0) <1$ and so $\pa_t u_1 \geq \min\{g(u_1),g(1)\}$ for $t\in (0,\eta)$ and $x=x_0$.
Arguing as in the proof of Lemma  \ref{lem:sat}, 
we deduce that either $u_1(0,x_0)<\frac 1 2$ or $u_1(0,x_0)\geq \frac 1 2$ in which case $\pa_t u_1(t,x_0) \geq \frac r 2 $ for $t\in (0,\eta)$ and thus
$u_1(0,x_0) \leq 1-\frac{\eta r}{2}$.
It follows that
$u_2(0,x_0)\leq u_1(0,x_0) \leq \max\{\frac 1 2 , 1-\frac{\eta r}{2}\}$ and since this bound is uniform in $x_0$, the claim follows.

Next, we want to show that
\begin{equation} \label{eq:LG}
\Lambda (t)<1  \quad\mbox{ for all $t>0$}
\end{equation}
by a contradiction argument.

Since $\Omega_1(t)$ is decreasing and $\pa_t u_2 \leq \sup_{[0,1]} g \leq L$, we have
$$ \Lambda(t+s) =\sup_{\overline {\Omega_1(t+s)}} u_2(t+s,x) \leq \sup_{\overline {\Omega_1(t)}} u_2(t+s,x) \leq \Lambda (t)+ Ls.$$
Assuming that \eqref{eq:LG} does not hold, this semi-continuity implies that there exists 
$t_0 := \inf\{ t\,;\, \Lambda (t)=1\}$ such that 
$$ t_0>0, \quad \Lambda (t_0)=1, \quad\mbox{ and } \Lambda (t) <1 \quad \forall t<t_0.$$
Since $\Lambda (t_0)=1$, \hb{there exists a sequence} $x_n\in \overline{\Omega_1(t_0)}$ such that $u_2(t_0,x_n)\to1$.
On the other hand, from $ \Lambda (t) <1$ for $t<t_0$ we infer that $u_2(t,x)<1$ whenever $v_1(t,x)<1$ and so 
\begin{equation}\label{eq:p1p2}
      {\mathbb 1}_{\{u_2 =1\}}(t,x) \leq {\mathbb 1}_{\{v_1 =1\}}(t,x) \qquad\mbox{ for all }x\in \R^d, \; t \in [0,t_0).
\end{equation}

In order to get a contradiction, we notice that $u_2$ satisfies
$$
\pa_t u_2(t,\hb{x_n}) \leq [g(u_2) + (g(1)-g(u_2) )K*{\mathbb 1}_{\{u_2 =1\}}  ] (t,\hb{x_n}) ,  \qquad \forall t<t_0.
$$
Furthermore, as $\hb{x_n}\in \overline{\Omega_1(t_0)}$, 
Lemma \ref{lem:sat} implies that $v_1(t,\hb{x_n})<1$ for all $t<t_0$ and so (using the definition of $v_1$ and \eqref{eq:p1p2}) we get
\begin{align*}
\pa_t v_1(t,\hb{x_n}) & = (1+\delta)[g(v_1) + (g(1)-g(v_1)) K*{\mathbb 1}_{\{v_1 =1\}} ] (t,\hb{x_n}) \\
& \geq (1+\delta)[g(v_1) + (g(1)-g(v_1)) K*{\mathbb 1}_{\{u_2 =1\}}](t,\hb{x_n})\qquad \forall t<t_0.
\end{align*}
We emphasize that this is where we use the condition on $g$ that $g(1)- g(u)\geq 0$.
Denoting  $h_2:=K*{\mathbb 1}_{\{u_2 =1\}}$, we infer that at $x=\hb{x_n}$ and for $t<t_0$ there holds:
\begin{align*}
\pa_t (v_1-u_2) 
& \geq (1-h_2)[g(v_1)-g(u_2)] + \delta (1-h_2) g(v_1) + \delta g(1) h_2 \\
& \geq \hb{q_n(t)}  (v_1-u_2) + \delta \Big[(1-h_2) g(v_1) + g(1) h_2 \Big]
\end{align*}
with $q_n(t)  = (1-h_2(t,x_n))\frac{g(v_1(t,x_n))-g(u_2(t,x_n))}{v_1(t,x_n)-u_2(t,x_n)}$ satisfying $|q_n|\leq L$. Moreover, \eqref{eq:f1} gives
$$
(1-h_2) g(v_1) + g(1) h_2 \geq \min \big( g(v_1),g(1)\big) \geq r\, v_1 \geq 0 .
$$

Since $v_1(0,x_n)-u_2(0,x_n) \geq 0 $, Gronwall's lemma applied to this differential inequality gives  $v_1(t,x_n)\geq u_2(t,x_n)$ for all $t\leq t_0$. We can then write (still at at $x=\hb{x_n}$ and for $t<t_0$)
\begin{align*}
\pa_t (v_1-u_2) 
& \geq -L  (v_1-u_2) + \delta r\, v_1 
\end{align*}
which yields:
\hb{$$
v_1(t_0,x_n) -u_2(t_0,x_n) \geq \delta\, r\,  e^{-Lt_0} \int_0^{t_0} e^{Ls}\, v_1(s,x_n)\, ds 
\geq \delta\, r \, e^{-Lt_0} \int_{t_0-\tau}^{t_0}  u_2(s,x_n)\, ds
$$
where we set $\tau:=\min\left\{ t_0,\frac{1}{4L}\right\}$.
Since $\pa_t u_2 \leq L$, we have $u_2(s,x_n)\geq u_2(t_0,x_n) - L\tau \geq u_2(t_0,x_n)-\frac 1 4$ for $s\in [t_0-\tau,t_0]$. Hence,
 for $n$ large enough (so that $ u_2(t_0,x_n)\geq \frac3 4$), we get
$$
v_1(t_0,x_n) -u_2(t_0,x_n) \geq \frac{\delta\, r\,\tau}{2}   e^{-Lt_0} >0 .
$$
This is a contradiction, since $v_1(t_0,x_n)-u_2(t_0,x_n)\leq 1-u_2(t_0,x_n)\to 0$.}
\medskip

We have thus proved \eqref{eq:LG}. This implies that \eqref{eq:p1p2} holds for all $t>0$ and so the nonlocal terms satisfy  $K*{\mathbb 1}_{\{v_1 =1\}}(t,x) \geq K*{\mathbb 1}_{\{u_2 =1\}}(t,x)$ for all $x\in \R^d$ and $t>0$.
We can then proceed as above to show that in the set $\{v_1(t,x)<1\}$, we must have $u_2(t,x)\leq v_1(t,x)$. Since this inequality clearly also holds when $v_1=1$, we deduce
$$ v_1(t,x) = u_1((1+\delta) t +\eta,x) \geq u_2(t,x) \quad\mbox{ for all $x\in \R^d$ and $t>0$} .$$
We can now pass to the limit $\delta\to0$ and $\eta\to0$ to get the result.
\end{proof}

\begin{remark}\label{rem:supersol}
\hb{The proof above holds if $u_1$ is a supersolution of \eqref{eq:main0} that is Lipschitz continuous with respect to $x$ (uniformly on bounded time intervals). }
Indeed, a supersolution is automatically non-decreasing in time, so that the sets $\Omega_1(t)$ are still non-increasing. Furthermore,  the proofs of Lemma \ref{lem:sat} and of \eqref{eq:v1u1} use nothing about $u_1$ beyond the lower bound $\pa_t u_1 \geq \min (g(u_1),g(1))$ in $\{u_1<1\}$ and the Lipschitz bound in $x$. 
In this more general setting, this Lipschitz bound must be assumed since it can no longer be deduced from Lemma \ref{lem:Lip}.
\end{remark}

We can now derive the invasion property, as 
Proposition \ref{prop:invasion} follows from the monotonicity in time and the following lemma.
\begin{lemma}\label{lem:stat}
    Assume that $g$ satisfies \eqref{eq:f1} and that $K>0$ in $B_\ell(0)$. Then
    \hb{any stationary solution of \eqref{eq:main0} satisfies either $u=0$ almost everywhere in $\R^d$, or $u\equiv 1$ in $\R^d$.}
\end{lemma}
\begin{proof}[Proof of Lemma \ref{lem:stat}]
Stationary solutions must satisfy (for all $x\in \R^d$)
$$\mbox{either } \; u(x)=1 \; \mbox{ or } \; g(u(x)) (1-K\ast {\mathbb 1}_{\{u=1\}}(x))+g(1)K\ast {\mathbb 1}_{\{u=1\}}(x)=0 .$$
Since $g(1)\neq 0$, the second condition is equivalent to
$$ K\ast {\mathbb 1}_{\{u=1\}}(x)=0 \mbox{ and } g(u(x))=0.$$
By \eqref{eq:f1}, this is the same as
$$ K\ast {\mathbb 1}_{\{u=1\}}(x)=0 \mbox{ and } u(x)=0.$$
In particular we see that a stationary solution $u(x)$ can only take the values $0$ or $1$ and so $u={\mathbb 1}_E$ for some set $E$.

Assume now that $E\neq \R^d$ and let $x_0\notin E$. 
Since $K*{\mathbb 1}_E(x_0)=0$, and $K>0$ in $B_\ell(0)$, we 
have:
\begin{equation}\label{eq:nullball}
|E\cap B_\ell (z) | = 0 \qquad \mbox{ for every } z \notin E.
\end{equation}
\hb{It is classical then that $E$ has measure zero. 
Indeed we can show that \eqref{eq:nullball} implies that  $|E\cap B_{k\ell}(x_0)|=0$ for all $k\geq 1$:
This holds for $k=1$ by \eqref{eq:nullball}. Assuming that it holds for some $k\geq1$, we note that  for any $w\in B_{(k+1)\ell}(x_0)$,
the set $B_\ell(w)\cap B_{k\ell}(x_0)$ contains some $z\notin E$ (since $E^c$ has full measure in $B_{k\ell}(x_0)$) for which we have $|E\cap B_\ell(z)|=0$ by \eqref{eq:nullball}. 
Since $z\in B_\ell(w)$, we have $w\in B_\ell(z)$. The balls $\{B_\ell(z)\}_{z\in E^c\cap B_{k\ell}(x_0)}$ thus cover $B_{(k+1)\ell}(x_0)$ and, extracting a countable subcover  yields $|E\cap B_{(k+1)\ell}(x_0)|=0$} and completes the proof.
\end{proof}

We can now prove the invasion property.

\begin{proof}[\hb{Proof of Proposition \ref{prop:invasion}}]
The monotonicity of $u(t,x)$ and $ p(t,x) = {\mathbb 1}_{\{u(t,x)=1\}}$ with respect to $t$ implies that there exist functions $u_\infty(x)$ and $p_\infty(x)$ such that
$$ u(t,x) \to u_\infty(x), \qquad p(t,x)\to p_\infty(x) \qquad \mbox{ as } t\to\infty,$$
the convergences holding pointwise in $x$ and in $L^q_{loc}(\R^d)$ for every $q<\infty$.

In order to show that $u_\infty$ is a stationary solution of \eqref{eq:main0}, we recall that $u(t,x)$ solves (see \eqref{eq:limit0})
$$
\pa_t u   =\big[ g(u )(1-K\ast  p ) +   g(1)K\ast p \big](1-p ) \qquad \mbox{ a.e. in } \R_+\times\R^d ,
$$
and we fix $x\in\R^d$. Since $t\mapsto u(t,x)$ is non-decreasing and bounded by $1$, there exists a sequence $t_n\to\infty$ such that $\pa_t u(t_n,x)\to 0$. Passing to the limit along $t_n$, we deduce
$$
\big[ g(u_\infty )(1-K\ast  p_\infty ) +  g(1) K\ast p_\infty )\big](1-p_\infty ) = 0 \qquad \mbox{ in } \R^d .
$$

It is easy to show that $p_\infty(x)=0$ in $\{u_\infty<1\}$. Furthermore, when $u_\infty(x)=1$, the equation above simplifies to $g(1)(1-p_\infty(x))=0$, that is $p_\infty(x)=1$.
We conclude that $p_\infty = {\mathbb 1}_{\{u_\infty=1\}}$ and that $u_\infty$ is a stationary solution of \eqref{eq:main0}.
Lemma \ref{lem:stat} and the assumption on the initial data implies that $u_\infty\equiv1$.
\end{proof}

Finally, we conclude this section with the proof of Lemma \ref{lem:invasion}:
\begin{proof}[Proof of Lemma \ref{lem:invasion}]
Equation \eqref{eq:main0} gives
$$\pa_t u \geq \min(g(u),g(1))\geq r u \quad \mbox{ as long as } u<1.$$
Given $x_0\in B_\delta(0)$, we have $u(t,x_0)\geq u^{in}(x_0)\geq \eta$ (recall that $u$ is monotone increasing in time).
It follows that as long as $u(t,x_0)<1$, we have
$\pa_t u(t,x_0) \geq   r\eta$. 
We deduce
$$ u(t,x)=1  \quad \mbox{ for all } x\in B_\delta(0)  \mbox{ and for } t\geq t_0$$
with $t_0= 1/(r\eta)$.
\hb{In turn, this implies that there exists $\eta_0>0$ (depending on $K$ and $\delta$) such that $K*{\mathbb 1}_{\{u=1\}}(t,x)\geq \eta_0$ for $t\geq t_0$ and $x\in B_\ell(0)$.
Indeed, $K*{\mathbb 1}_{\{u=1\}}(t,x)\geq K*{\mathbb 1}_{B_\delta(0)}(x)$ for $t\geq t_0$ and this quantity is positive for every $x\in \overline{B_\ell(0)}$. Since $x\mapsto K*{\mathbb 1}_{B_\delta(0)}(x)$ is continuous, its infimum  over   $\overline{B_\ell(0)}$ is positive. Going back to \eqref{eq:main0}, we now get:
\begin{align*}
\pa_t u (t,x)& = g(u(t,x))(1-K* \mathbb{1}_{\{u=1\}}(t,x)) + g(1)K* \mathbb{1}_{\{u=1\}}(t,x) \\
& \geq g(1)K* \mathbb{1}_{\{u=1\}}(t,x)\\
& \geq g(1)\eta_0 \geq r \eta_0
\end{align*}
for $x\in B_\ell(0)$ and $t\geq t_0$, as long as $u(t,x)<1$.
We deduce that 
$$ u(t,x)=1 \quad \mbox{ in } B_\ell(0), \mbox{ for } t\geq t_0+t_1$$
with $t_1=1/(r \eta_0)$.
This argument can now be iterated, with the same constant $\eta_0$ to show that if $u(t,\cdot)=1$ in $B_R(0)$ for $t\geq \bar t$ and $R\geq \delta$, then $u(t,\cdot)=1$ in $B_{R+\ell-\delta}(0)$ for $t\geq \bar t+t_1$.}

Starting from $R=\delta$ and $\bar t = t_0$, we find
$$ u(t,x)=1 \quad \mbox{ in } B_{\delta+k(\ell-\delta)}(0), \mbox{ for } t\geq t_0+kt_1$$
for every $k\geq 0$, which yields \eqref{eq:ndgsp} with $\underline c = \frac{\ell-\delta}{t_1}$ (upon replacing $t_0$ by $t_0+t_1$).
\end{proof}

\section{Traveling waves: Proof of Theorem \ref{thm:TW}}\label{sec:TW}
\hb{Traveling} wave solutions are key to understand the dynamics of the convergence of $u(t,x)$ to the stationary state $1$ as $t\to\infty$. We now discuss and characterize the existence of such solutions for \eqref{eq:main0}.

 Planar traveling waves (or traveling waves for short) are solutions of the form $u(t,x)=\phi(x\cdot e-ct)$ with $e\in \mathbb S^{d-1}$, $c\in \R$ and 
$$ \lim_{s\to -\infty} \phi(s) = 1, \qquad  \lim_{s\to +\infty} \phi(s) = 0. $$
Since all solutions of \eqref{eq:main0} are monotone increasing and Lipschitz continuous with respect to $t$, the profile function $s\mapsto \phi(s)$ is  monotone decreasing and Lipschitz continuous.
In particular, there exists $s_0$ such that $\{\phi=1\}=(-\infty,s_0]$, and up to translation, we can always assume that $s_0=0$, which we will do from now on.
Traveling wave solutions of equation \eqref{eq:main0} are found as solutions of:
$$ - c \phi' (x\cdot e)= g(\phi (x\cdot e))  + (g(1)-g(\phi(x\cdot e)))K\ast {\mathbb 1}_{\{x \cdot e\leq 0\}}(x) \quad  \mbox{ for } x\cdot e >0.
$$
Given $e$, $|e|=1$, let  $H(e)=\{ x \; ; \; x\cdot e <0\}$.
The equation for traveling wave solutions involves the function $h:\R\to [0,1]$ which is defined as
\begin{equation}\label{eq:defh} 
h(s) :=   K\ast {\mathbb 1}_{H(e)}( s e).
\end{equation}
Then, for any $x$, there holds
\begin{equation*}
K\ast {\mathbb 1}_{H(e)}( x) = h( x\cdot e).
\end{equation*}
Indeed, up to a rotation, we can assume that $e=e_1$. Then, splitting the coordinates into $x=(x_1, x')$ with $x'=(x_2, \ldots, x_d)$, we can write
\begin{align*}
 K\ast {\mathbb 1}_{\{ x_1<0\}}(x) 
 &  = \int_{\{y_1<0\}} K(x-y)\, dy \\
 & =  \int_{-\infty}^0 \int_{\R^{d-1}} K(x_1-y_1,x'-y')\, dy' \, dy_1 \\
 & = \int_{-\infty}^0 \int_{\R^{d-1}} K(x_1-y_1,y')\, dy' \, dy_1 
\end{align*}
which is a function of $x_1$ only and coincides with $h(x_1)$.
This computation shows that
$$
h(s)= \int_{\{ y_1 >s\}} K(y) dy.
$$
The fact that $K$ is radially symmetric implies that the function $h(s)$ defined in \eqref{eq:defh}, does not depend on the particular choice of $e$. Without that assumption it
would depend on the direction $e$.

The function $h$ is determined from  $K$ and can be computed explicitly for simple choices of the kernel $K$. It satisfies
\begin{align*}
   & h:\R \to [0,1] \mbox{ is non-increasing }, \quad h(0)=1/2, \\
    &  \; h(s)=0 \mbox{ for all } s\geq \ell,\quad
    h(s)=1 \mbox{ for all }
s\leq -\ell.
\end{align*}

Traveling waves are thus determined by the following equation. 
\begin{lemma}
Up to translation, we can assume that the profile $\phi$ of a traveling wave with speed $c>0$,  satisfies $\phi(s)=1$ for all $s\leq 0$ and solves the following initial value problem:
\begin{equation}\label{eq:c}
 \phi(0)=1 ,\qquad   - c \phi' (s)= 
g(\phi (s))  + (g(1)-g(\phi(s)))h(s) \quad s\in [0,\infty)  .
 \end{equation}
 with $h$ defined by \eqref{eq:defh}.
\end{lemma}
Note that traveling wave profiles $\phi$ do not depend on $e$ but do depend on the dimension $d$ and $K$ because of the coefficient $h(s)$.

To consider the general initial value problem  \eqref{eq:c}, we extend $g(s)$ outside $[0,1]$ by setting $g(s)=0$ for $s\leq 0$, and $g(s)= g(1)$ for $s\geq 1$.
We denote by $\phi_c$ the unique solution of  \eqref{eq:c} (for $c>0$). We have to identify the values of $c$ for which $\phi_c(s)$ remains non-negative and converges to $0$ at $+\infty$.
This last condition is actually easy to obtain:
Since $h(s)=0$ in the interval $[\ell,+\infty)$, \eqref{eq:c} writes
$$ - c \phi_c' (s)= g(\phi_c (s)) , \qtext{for} s\geq \ell.$$
Hence $\phi_c(s)$ takes non-negative values in $[\ell,+\hb{\infty)}$ if and only if $\phi_c(\ell)\in [0,1]$ in which case it satisfies $\lim_{s\to \infty}\phi_c(s)=0$. In addition, $\phi_c(\ell)=0$ implies $\phi_c(s)=0$ for all $s\geq \ell$ (which corresponds to the traveling wave with minimal speed).
\medskip

Proving the existence of a traveling wave solution with speed $c$ is thus reduced to making sure that the solution of \eqref{eq:c} takes non-negative values in $[0,\ell]$.
We note that if there exists $s_0\in (0,\ell)$ such that  $\phi_c(s_0)=0$, then $-c\phi_c'(s_0) = g(1) h(s_0)>0$ and thus $\phi_c'(s_0)<0$. This implies that for any $c>0$, $\phi_c$ can  vanish at most once in $(0,\ell)$ and when it does we will have $\phi_c(\ell)<0$. So we only need to study the sign of $\phi_c(\ell)$ as a function of $c$.
\medskip

The discussion above shows that Theorem \ref{thm:TW} is an immediate consequence of the following.
\begin{proposition}\label{prop:c}
Let $\phi_c$ be the unique solution of  \eqref{eq:c}.
There exists a unique $c^*>0$ such that $\phi_{c^*}(\ell) =0$. \hb{It characterizes the sign of $\phi_c(\ell)$: one has
$\phi_c(\ell)>0$ if $c>c^*$, and $\phi_c(\ell)<0$ if $c<c^*$.}
Furthermore, $\phi_c(s)>0$ in $(0,\ell)$ for all $c\geq c^*$.
\end{proposition}

\begin{figure}[ht]
  \centering
  \includegraphics[width=0.62\textwidth]{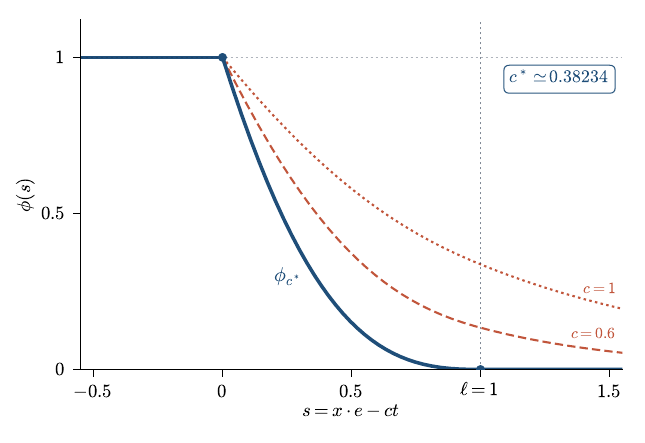}
  \caption{Traveling wave profiles in dimension $d=2$ when
  $K=\frac{1}{|B_\ell|}\mathbb 1_{B_\ell}$ with $\ell=1$ and $g(u)=u$. 
  The minimal speed is
  $c^*\simeq0.382$. The profile $\phi_{c^*}$ (solid blue line) vanishes
  for  $s\geq \ell$. For comparison, the profiles
  $\phi_c$ with $c=0.6$ and $c=1$ (dashed red lines) remain positive on the whole line.}
  \label{fig:tw}
\end{figure}

The proof of this proposition relies on the monotonicity of the map $c\mapsto \phi_c(s)$ in the set where $\phi_c\geq 0$ as given by the following lemma:
\begin{lemma}\label{lem:mono} 
    If $\phi_{c_0}(s)\geq 0$ for $s\in [0,s_0]$ with $s_0\leq \ell$, then for all $c>c_0$ we have 
    $$ \phi_c(s) > \phi_{c_0} (s) \quad\mbox{ for all } s\in (0,s_0].$$
\end{lemma}

\begin{proof}[Proof of Lemma \ref{lem:mono}]
    \hb{By assumption,} $\phi_{c_0}$ is non-negative in $[0,s_0]$. Then, 
 Equation \eqref{eq:c} implies that $\phi_{c_0}'(s)< 0$ for all $s\in[0, s_0)$,
 as we can write the right hand side of \eqref{eq:c}  in the form 
$g(\phi (s)) (1-h(s)) + g(1) h(s)$.
For $s>0$ small, the equation shows that 
$\phi_c(s)>\phi_{c_0}(s)$. If there were a first value $0< s_1<s_0$ where $\phi_c$ and $\phi_{c_0}$ would cross, we would infer from the equation that 
$\phi_c'(s_1) >\phi_{c_0}'(s_1)$, which is impossible. Therefore, $\phi_c > \phi_{c_0}$ in $(0, s_0)$. The same reasoning applies up to $s=s_0$, at least when $\phi_{c_0}(s_0)>0$, since in this case, $\phi_{c_0}'(s_0)<0$. We only need to also discard the case $\phi_c(s_0)=\phi_{c_0}(s_0)=0$. If this were to happen, setting $w:= \phi_c - \phi_{c_0}$ we would have $w>0$ in $(0, s_0)$ and $w(s_0)=0$. Using the equation we infer that 
\begin{equation*}
    - c w'(s) < \hb{[}g(\phi_c(s)\hb{)}   - g(\phi_{c_0}(s))\hb{]} (1-h(s)).
\end{equation*}
Using the Lipschitz character of $g$, we infer that there is a constant $m>0$ such that 
\begin{equation*}
    c w'(s) + m w (s) >0 \qtext{for} s\in [0, s_0).
\end{equation*}
This would say that $s\mapsto (e^{m s/ c} w )$ is increasing, a clear contradiction with $w(s_0)=0$ and $w(s)>0$ for $s<s_0$. Therefore, when $\phi_{c_0}(s_0)=0$, we also get $\phi_c(s_0) >0$. 
\end{proof}

\begin{proof}[Proof of Proposition \ref{prop:c}]
The proposition follows from Lemma \ref{lem:mono} and the following simple observations:
\begin{itemize}
\item If $c>    \ell \sup g$, then $\phi_c(\ell)>0$.
Indeed, we have $ -c \phi_c'(s)  \leq \sup g$ for all $s>0$ and so $\hb{\phi_c}(\ell)\geq 1- (\ell/c) \sup g$.
\item If $c<   g(1) \int_0^\ell h(s)\, ds$, then $\phi_c(\ell) <0$. Indeed, if $\phi_c(\ell) \geq 0$ , then we have   
 $-c\phi_{c}'(s) \geq g(1) h(s)$ for $s\in(0,\ell)$ and so $\phi_c(\ell) \leq 1-\frac{  g(1)}{c} \int_0^\ell h(s)\, ds <0$, a contradiction.
 \end{itemize}

We can thus define 
$$ c^* := \inf\{ c\,;\, \phi_c(\ell)>0\},
\quad 0 < c^* < \infty.
$$

Lemma \ref{lem:mono} implies that 
$\phi_c(\ell)>0$ for all $c>c^*$ and $\phi_c(\ell)<0$ for all $c<c^*$.
From the continuity of the map $c\mapsto \phi_c(\ell)$, it follows that  $\phi_{c^*}(\ell)=0$ which completes the proof.
\end{proof}
Note that the arguments above provide the following lower and upper bounds on the minimal speed~$c^*$:
$$
g(1) \int_0^\ell h(s)\, ds\leq c^* \leq \ell \sup g.
$$

\begin{remark}
    It should be noted that even in the case when $K$ is the normalized characteristic function of a ball, the planar \hb{traveling} fronts $\phi_c$ which are functions of one variable, actually depend on the dimension $d$. This is quite different from the classical framework of reaction-diffusion equations.
\end{remark}

\section{Spreading of compactly supported solutions}\label{sec:spreading}
\subsection{Spreading of the support and the saturated set}
The existence of traveling wave solutions with finite speed, together with the comparison principle (Proposition \ref{prop:comp}), imply that compactly supported solutions propagate with finite speed, as stated in Corollary \ref{cor:supp}.
In this section, we will prove that this spreading occurs with  speed $c^*$. 
\medskip

We note that with general reaction-diffusion \hb{equations}, it is important to distinguish the spreading of the support of the solution (which might take place instantly) and the spreading of the invasion front (characterized by traveling wave solutions).
For our problem, this distinction is not relevant. Indeed, we can show that the support of $u(t,\cdot)$ and the saturated set remain close to each other for all time, as stated in Proposition \ref{prop:support} which we prove below:

\begin{proof} [Proof of Proposition \ref{prop:support}]
Equation \eqref{eq:main0} together with assumption \eqref{eq:f1} imply
$$ \pa_t u \leq L( u  +  K*{\mathbb 1}_{\{u =1\}}). $$
In particular we have 
$$ 
u(t,x) \leq u(0,x) \hb{e^{Lt}} + \int_0^t e^{L(t-s)} K*{\mathbb 1}_{\{u =1\}}(s,x)\, ds.
$$
This inequality implies that if $x_0\in \R^d$ is such that $u(0,x_0)=0$ and $u(t_0,x_0) >0$, then we must have 
$$\int_0^{t_0} e^{L(t_0-s)} K*{\mathbb 1}_{\{u =1\}}(s,x_0)\hb{\, ds}>0.$$
Since the function $s\to K*{\mathbb 1}_{\{u =1\}}(s,x_0)$ is non-decreasing, this means that $K*{\mathbb 1}_{\{u =1\}}(t_0,x_0)>0$, and the fact that $K$ is supported in $B_\ell(0)$ yields
$$|\supp \, {\mathbb 1}_{\{u =1\}}(t_0,\cdot) \cap B_\ell (x_0) |  = |S(t_0) \cap B_\ell (x_0) |   >0$$
which implies that $x_0\in S(\hb{t_0})+B_\ell$.

We just proved that given $(t_0,x_0)$ such that $u(t_0,x_0) >0$, we must have either $u(0,x_0)>0$ or 
$x_0\in S(t_0)+B_\ell(0)$. The result follows.
\end{proof}

\subsection{Speed of spreading - Proof of Theorem \ref{thm:asymptotic-spreading}}

The first part of Theorem \ref{thm:asymptotic-spreading}, follows immediately from the comparison principle and the existence of a traveling wave  solution with speed $c^*$.
We denote by $\phi_*$ the profile of the traveling wave with minimal speed $c^*$ constructed in the previous section.
Since $u^{in}$ is compactly supported and $\phi_*(s)=1$ for $s\leq 0$, it is clear that there exists a constant $M$
such that for any  direction $e\in\mathbb S^{d-1}$ 
$$ 
u^{in}(x) \leq \phi_*(x\cdot e- M) \mbox{ for all }x\in \R^d.
$$
The comparison principle then implies
$$ u(t,x) \leq \phi_*(x\cdot e-c^*t - M) \mbox{ for all }x\in \R^d, \; t>0 , \mbox{ and all } e\in\mathbb S^{d-1}.$$
From this we infer that
$$
u(t,x) \leq \inf_{e\in\mathbb S^{d-1}} \phi_*(x\cdot e-c^*t - M)=
 \phi_*(|x|-c^*t - M)
$$
which yields the limit \eqref{eq:spreading-upper} as
$$
\limsup_{t\to\infty} \{ \sup_{|x| \geq ct} u(t,x)\} \leq \limsup_{t\to\infty}  \phi_*((c-c^*)t - M) = 0 \quad \mbox{ for all
} c>c^*.
$$
This result means that $c^*$ is an upper bound for the speed of spreading for a solution starting with a compactly supported initial datum. 

\medskip
We now establish that $c^*$ is  a lower bound for the speed of spreading by proving the more precise statement~\eqref{eq:ballinvasion}.
To this end, we construct compactly supported radial subsolutions propagating with any prescribed speed $c<c^*$.  

\medskip

Fix $c<c^*$ and define
\[
    v(t,x):=\phi_*(|x|-ct-R)
\]
which satisfies:
\[
    \operatorname{supp}v(t,\cdot)= B_{R+ct+\ell},
    \qquad
    \{v(t,\cdot)=1\}=B_{R+ct}.
\]
Let
\[
    s:=|x|-ct-R.
\]
On the set $\{v<1\}$ we have $s>0$, and the traveling-wave equation
\eqref{eq:c} gives
\begin{align*}
    \partial_t v(t,x)
      &=-c\phi_*'(s)\\
      &=g(v(t,x))
        +(g(1)-g(v(t,x)))h(s)
        +(c^*-c)\phi_*'(s).
\end{align*}
The same equation and the monotonicity of $\phi_*$ imply
\[
    -c^*\phi_*'(s)
       =g(\phi_*(s))
        +(g(1)-g(\phi_*(s)))h(s)
       \ge (g(1)-g(\phi_*(s)))h(s).
\]
and so
$$\phi_*'(s) \leq -\frac{1}{c^*}(g(1)-g(v(t,x)))h(s) .$$
It follows that
\begin{equation*}
     \partial_t v(t,x)
\le g(v(t,x))
       +\left( 1 - \frac{c^*-c}{c^*}\right) (g(1)-g(v(t,x)))h(s).
\end{equation*}
Letting 
\[
    \delta:=\frac{c^*-c}{c^*}\in(0,1).
\]
we get
\begin{equation} \label{eq:delta}
    \partial_t v(t,x)
\le g(v(t,x))
       +(1-\delta)(g(1)-g(v(t,x)))h(s).
\end{equation}

In order to obtain a subsolution, we need to compare $(1- \delta) h(s)$ with the
corresponding coefficient in the equation~($\star$), namely $K \ast \mathbb{1}_{\{v=1 \}} $. To this end we are going to use that $\delta>0$.
Recall from \eqref{eq:defh} that
\[
    h(s)=K*\mathbf 1_{\{x_1<0\}}(se_1)
        =\int_{\{x_1>s\}}K(x)\,dx .
\]
This comparison relies on the following geometric ingredient that yields a comparison between the dispersal when the saturated region  is a large ball and that generated by the limiting half space when the radius goes to infinity. Indeed, the former is what describes spreading from a confined initial state whereas the latter represents the spreading for planar traveling fronts.

\begin{lemma}[Large balls versus half-spaces]\label{lem:large-ball-halfspace}
Let $K\ge0$ be radially symmetric and supported in $B_\ell$.  There is a
constant $C_d>0$, depending only on the dimension, such that, for every
$R\ge2\ell$,
\begin{equation}\label{eq:large-ball-relative}
    K*\mathbf 1_{B_R}(x)
    \ge
    \left(1-C_d\frac{\ell}{R}\right)
    h(|x|-R)
    \qquad\text{for all }|x|\ge R.
\end{equation}
In particular, for every $\delta>0$ there exists $R_\delta>0$ such that
\begin{equation}\label{eq:large-ball-delta}
    (1-\delta)h(|x|-R)
    \le K*\mathbf 1_{B_R}(x)
    \qquad\text{for all }R\ge R_\delta,\quad |x|\ge R.
\end{equation}
\end{lemma}
Postponing the proof of this lemma to the end of this section we first complete the proof of \eqref{eq:ballinvasion}. 

Given $c$ and $\delta= \delta (c)$ as above, the lemma, applied  with radius $R+ct$, shows that there exists $R= R(c)>0$ such that for any $R\geq R(c)$ and for all $t\geq 0$, we have
\begin{equation}
    \partial_t v(t,x)
\le g(v(t,x))
       +(g(1)-g(v(t,x)))
          K*\mathbf 1_{B_{R+ct}}(x) \label{eq:radial-subsolution}
\end{equation}
Since $\{v(t,\cdot)=1\}=B_{R+ct}$, the right-hand side of
\eqref{eq:radial-subsolution} is precisely
\[
    g(v)+(g(1)-g(v))K*\mathbf 1_{\{v=1\}}
\]
on $\{v<1\}$. We also know that  $\partial_t v=0$ on $\{v=1\}$.  Therefore, $v$ is a
subsolution of \eqref{eq:main0}\hb{, the required regularity being immediate since $\phi_*$ is Lipschitz continuous}.

We now compare this subsolution with $u$.  By Lemma~\ref{lem:invasion}, there exists
$\bar t>0$ (depending on $c$) such that
\[
    u(\bar t,x)=1
    \qquad\text{for every }x\in B_{R+\ell},
\]
and in view of $\operatorname{supp}v(0,\cdot)\subset B_{R+\ell}$, this implies that
\[
    v(0,x)=\phi_*(|x|-R)\le u(\bar t,x)
    \qquad\text{for every }x\in\mathbb R^d.
\]
The comparison principle then yields
\begin{equation}\label{eq:shifted-comparison}
    \phi_*(|x|-ct-R)
       =v(t,x)
       \le u(\bar t+t,x)
    \qquad\text{for every }x\in\mathbb R^d,\quad t\ge0.
\end{equation}
Since $\phi_*=1$ on $(-\infty,0]$ and $u\le 1$, we infer that
$$u(\bar t+t,x)=1 \quad \mbox{ for all } x \mbox{ with } |x|<ct+R.$$
In other words, we proved that
\[
    B_{R+ct}\subset\{u(\bar t+t,\cdot)=1\}
    \qquad\text{for every }t\ge0.
\]
\hb{or, equivalently, $B_{R+c(t-\bar t)}\subset\{u(t,\cdot)=1\}$ for every $t\geq \bar t$.
To conclude, we fix $c<c^*$ and apply the construction above with some $c'\in (c,c^*)$, which provides $R=R(c')$ and $\bar t=\bar t(c')$ as above. Since
$$
R+c'(t-\bar t) \geq ct \qquad \mbox{ for all } t\geq \frac{c'\bar t}{c'-c},
$$
we obtain \eqref{eq:ballinvasion} with $T_c:=\max\left\{ \bar t, \frac{c'\bar t}{c'-c}\right\}$.}
The proof of Theorem~\ref{thm:asymptotic-spreading} is thereby complete, except for the proof of the lemma to which we now turn.
\medskip
 

\medskip

\subsection{Proof of Lemma~\ref{lem:large-ball-halfspace}}
Since it has a simple geometric interpretation that is easily seen in the 
particular case where $K=\frac{1}{|B_\ell| } \mathbb{1}_{B_\ell}$, we start by providing a sketch of the proof in that case and then derive the general case.

\begin{proof}[Proof of Lemma \ref{lem:large-ball-halfspace} when $K=\frac{1}{|B_\ell| } \mathbb{1}_{B_\ell}$]
The radial symmetry of $K$ implies that it is enough to prove the inequality when $x=(x_1,0)\in [R,+\infty)\times\R^{d-1}$.
We then have $|x|-R=x_1-R$ and we can write
$$
h(x_1-R) = \int_{y_1< 0} K(x_1-R-y_1,y')\, dy = \frac{|H\cap B_\ell(x_1-R,0)|}{|B_\ell|} $$
where $H  $ is the half space $H =\{ (y_1,y')\in \R^d\, ;\, y_1<0\}$.

On the other hand, we have
$$K\ast {\mathbb 1}_{B_R}(x) = \int_{|y|<R} K(x-y) \, dy  = 
\int_{D_R} K(x_1-R-y_1,y')\, dy   = \frac{|D_R\cap B_\ell(x_1-R,0)|}{|B_\ell|} 
$$
where $D_R$ is the disc centered at $(-R,0)$ with radius $R$:
$$ D_R= \{(y_1,y')\, ;\, (y_1+R)^2+{y'}^2\leq R^2\}.$$
Inequality \eqref{eq:large-ball-relative} is thus equivalent to
the geometric inequality
\begin{equation}\label{eq:geo}
 |D_R\cap B_\ell(s,0)| 
    \ge
    \left(1-C_d\frac{\ell}{R}\right)
   |H\cap B_\ell(s,0)| 
    \qquad\text{for all } s\geq 0.  
\end{equation}
We note that $D_R \subset H $, and $D_R\cap B_{\ell}(s,0)\to H\cap B_{\ell}(s,0)$ when $R\to\infty$  (this is illustrated in Figure~\ref{fig:limR}). Inequality \eqref{eq:geo} simply quantifies that fact.
\begin{figure}
\begin{center}
\scalebox{.8}{
\begin{tikzpicture}
\fill [pattern=north west lines,pattern color=red] (0,1.6) arc (126.869897646:233.130102354:2);
\fill [white] (-0.20645,1.42193) arc (134.686417733:225.313582267:2) (-0.20645,1.42193) arc (16.5221481395:-16.5221481395:5);
\fill [pattern=north east lines,pattern color=blue] (-0.20645,1.42193) arc (134.686417733:225.313582267:2) (-0.20645,1.42193) arc (16.5221481395:-16.5221481395:5);
\draw [thick,->] (-4,0) -- (4,0);
\draw [thick,->] (0,-4) -- (0,4);
\draw [thick] (1.2,0) circle (2) node [cross] {} node [anchor=north] {\large$s$} (-1.46446609407,3.53553390593) arc (45:-45:5) (0.775,0.905195) node [anchor=south west] {\Large$\ell$};
\draw [thick,->] (-4.06667,0.88192) -- (-1.25,3.30719);
\draw (-2.66667,2.20479) node [anchor=north west] {\large$R$};
\draw [thick,->] (1.2,0) -- (0.35,1.81039);
\end{tikzpicture}
}
\end{center}
\caption{The spherical cap $C=H\cap B_\ell(s)$ (blue and red regions) corresponds to the planar traveling front.
The blue region $E=D_R\cap B_\ell(s)$ corresponds to the spherical symmetric solution.
The  region in red  is $C\setminus E$.
Inequality \eqref{eq:geo1} states that
 the ratio $\frac{|C\setminus E|}{|C|}$ is $\mathcal O(\ell/R)$.}
 \label{fig:limR}
\end{figure}

To prove \eqref{eq:geo}, we first remark that both sides vanish when $s\geq \ell$, while when $s<\ell$ it is equivalent to
\begin{equation}\label{eq:geo1}
\frac{|(H\setminus D_R)\cap B_{\ell}(s,0)|}{|H\cap B_{\ell}(s,0)|} 
\leq C_d \frac{ \ell}{R} .
\end{equation}
Finally, after rescaling, it is enough to prove this inequality when $\ell=1$, that is:
\begin{equation}\label{eq:geo2}
\frac{|(H\setminus D_R)\cap B_{1}(s,0)|}{|H\cap B_{1}(s,0)|} 
\leq \frac{C_d }{R} \qquad 0\leq s\leq 1.
\end{equation}
This inequality now follows from simple geometric considerations  that we illustrate in Figure \ref{fig:geo}.
\begin{figure}
\begin{center}
\scalebox{.8}{
\begin{tikzpicture}
\fill [white] (-0.20645,1.42193) arc (134.686417733:225.313582267:2) (-0.20645,1.42193) arc (16.5221481395:-16.5221481395:5);
\filldraw [pattern=north east lines, draw=black, pattern color=blue] (-.8,0) -- (0,0) -- (0,1.6); 
\filldraw [pattern=north east lines, draw=black, pattern color=blue] (-.8,0) -- (0,0) -- (0,-1.6); 
\draw [thick] (-.8,0) -- (0,1.6);
\draw [thick] (-.8,0) -- (0,-1.6);
\draw [thick,->] (-3,0) -- (4,0);
\draw [thick,->] (0,-4) -- (0,4);
\draw [thick] (1.2,0) circle (2) node [cross] {} node [anchor=north] {\large$s$};
\draw [thick,->] (1.2,0) -- (0,1.6);
\draw (.6,1.2) node [anchor=north west] {\large$1$};
\draw (0,.8) node [anchor=north west] {\large$a$};
\draw (-.6,-.1) node [anchor=north west] {\large$\delta$};
\end{tikzpicture}
}
\scalebox{.8}{
\begin{tikzpicture}
\fill [pattern=north west lines,pattern color=red] (0,1.6) -- (0,0) -- (-.235,1.52);
\fill [pattern=north west lines,pattern color=red] (0,-1.6) -- (0,0) -- (-.235,-1.52);
\draw [thick,->] (-6,0) -- (4,0);
\draw [thick,->] (0,-4) -- (0,4);
\draw [thick] (1.2,0) circle (2) node [cross] {} node [anchor=north] {\large$s$} (-1.46446609407,3.53553390593) arc (45:-45:5); 
\draw [thick] (-5,0) -- (0,1.6);
\draw [thick] (-5,0) -- (0,-1.6);
\draw [thick] (0,0) -- (-.235,1.52);
\draw [thick] (0,0) -- (-.235,-1.52);
\draw (-5,0) node [anchor=north] {\large$-R$};
\draw (0,1) node [anchor=north west] {\large$a$};
\draw (-.4,2.1) node [anchor=north west] {\large$b$};
\end{tikzpicture}
}
\end{center}
\caption{Sketch of the proof of \eqref{eq:geo2}. On the left, the blue conical region is smaller than the spherical cap $H\cap B_1(s,0)$ and has volume (or area if $d=2$) $c_d a^{d-1} \delta$. On the right, the red region contains $(H\setminus D_R)\cap B_1(s,0)$ and has volume $\leq C_d a^{d-1}b$. Using the fact that $b=\sqrt{R^2+a^2}-R\leq \frac{a^2}{2R}$ and $a^2  = 1-(1-\delta)^2 \leq 2\delta$, gives \eqref{eq:geo2}.}
\label{fig:geo}
\end{figure}
\end{proof}

\medskip

\begin{proof}[Proof of Lemma \ref{lem:large-ball-halfspace} for general $K$]
By radial symmetry, it is enough to consider
\[
    x=(R+s)e_1,\qquad s\ge0.
\]
We can then write:
\begin{align*}
    h(s)
      &=\int_{H_s}K(z)\,dz,
      &H_s&:=\{z\in\mathbb R^d\, ;\, z_1>s\},\\
    K*\mathbf 1_{B_R}((R+s)e_1)
      &=\int_{D_{R,s}}K(z)\,dz,
      &D_{R,s}&:=\{z\in\mathbb R^d\,;\,              |(R+s)e_1-z|<R\}
\end{align*}
where the condition $|(R+s)e_1-z|<R$ is equivalent to
\begin{equation}\label{eq:ballcondition}
    z_1>s+\frac{|z|^2-s^2}{2(R+s)}.
\end{equation}
Note that $D_{R,s}\subset H_s$ and both integrals vanish when $s\ge\ell$.  

With these notations, \eqref{eq:large-ball-relative} can be written as
$$
\int_{D_{R,s}}K(z)\,dz \geq \left(1-C_d\frac{\ell}{R} \right)\int_{H_s}K(z)\,dz,\qquad \mbox{ for all } 0\leq s\leq \ell,
$$
which  is equivalent to
\begin{equation}\label{eq:lost-mass}
    \int_{H_s\setminus D_{R,s}}K(z)\,dz
    \le C_d\frac{\ell}{R}\int_{H_s}K(z)\,dz \qquad \mbox{ for all } 0\le s<\ell. 
\end{equation}

The inequality is straightforward in dimension $1$ so we can take $d\geq 2$.
Write $K(z)=k(|z|)$ \hb{(as in \eqref{eq:K2})} and use polar coordinates
$z=r\omega$, with $r>0$ and $\omega\in\mathbb S^{d-1}$.  For fixed
$r>s$, set
\[
    q:=\frac{s}{r}\in[0,1).
\]
The intersection of the sphere of radius $r$ with the half-space $H_s$ corresponds to 
$$ S_r \cap H_s = \{ (r,\omega) \, ;\, \omega_1>q\}.$$
Similarly, using \eqref{eq:ballcondition}, we see that the intersection of  the sphere of radius $r$ with the ball $D_{R,s}$
is 
$$S_r \cap D_{R,s} = \{ (r,\omega)\, ;\, \omega_1>q+\eta\} \quad \mbox{ with } \eta:=\frac{r^2-s^2}{2(R+s)r}.$$
Thus the part lost on this sphere is the strip
\[
S_r \cap (H_s\setminus D_{R,s}) =    \{ (r,\omega)\, ;\,    q<\omega_1\le q+\eta\}.
\]
Moreover, since $s\leq r\leq \ell$, we have
\begin{equation}\label{eq:theta-definition}
    \theta:=\frac{\eta}{1-q}
       =\frac{r+s}{2(R+s)}
       \le \frac{\ell}{R}.
\end{equation}
For $R\ge2\ell$, we deduce $0\le\theta\le1/2$, and hence
$q+\eta<1$.

Let $\sigma$ denote surface measure on $\mathbb S^{d-1}$.  
The standard
slicing formula on the sphere gives, with
$\alpha=(d-3)/2>-1$,
\begin{align*}
\mathcal H^{\hb{d}-1} ( S_r \cap H_s)  
= \sigma\{\omega_1>q\}
   &=|\mathbb S^{d-2}|
       \int_q^1(1-u^2)^\alpha\,du,\\
\mathcal H^{\hb{d}-1}  (S_r\cap (H_s\setminus D_{R,s})) 
=  \sigma\{q<\omega_1\le q+\eta\}
   &=|\mathbb S^{d-2}|
       \int_q^{q+\eta}(1-u^2)^\alpha\,du .
\end{align*}
For $u\in[0,1]$, the quantities
$(1-u^2) $ and $(1-u)$ are comparable. Therefore
\begin{align*}
\frac{\mathcal H^{\hb{d}-1}  (S_r \cap (H_s\setminus D_{R,s}))}{\mathcal H^{\hb{d}-1} ( S_r \cap H_s)  }
 &\le C_d
   \frac{\displaystyle\int_q^{q+\eta}(1-u)^\alpha\,du}
        {\displaystyle\int_q^1(1-u)^\alpha\,du}\\
 &=C_d\left[1-(1-\theta)^{\alpha+1}\right]
 \le C_d\theta
 \le C_d\frac{\ell}{R}.
\end{align*}
\hb{where} we used the fact that $0\le\theta\le1/2$ and
$\alpha+1=(d-1)/2>0$. 

Finally, we write
\begin{align*}
\int_{H_s\setminus D_{R,s}} K(z)\, dz & = \int_s^\ell k(r) r^{d-1}  \mathcal H^{\hb{d}-1} ( S_r \cap (H_s\setminus D_{R,s})) \, dr \\
& \leq 
C_d \frac{\ell}{R}
\int_s^\ell k(r) r^{d-1}  \mathcal H^{\hb{d}-1} ( S_r \cap H_s) \, dr \\
&  \leq C_d \frac{\ell}{R} \int_{H_s } K(z)\, dz
\end{align*}
which is \eqref{eq:lost-mass}, whence \eqref{eq:large-ball-relative} follows.

Then, taking for example
$    R_\delta:=\max\left\{2\ell,\frac{C_d\ell}{\delta}\right\}$
yields \eqref{eq:large-ball-delta}.
\end{proof}

\section{Conclusion and perspectives}\label{sec:Conclusions}

We introduced a general class of models~\eqref{eq:main1} aimed at describing how a population can expand its range without individual movement, through the effect of congestion. This framework provides a description of spatial growth suited to populations of immotile cells and captures several key features observed experimentally. We established well-posedness of the associated initial value problem.
We then investigated in detail a singular limit of this model, given by equation~\hb{\eqref{eq:main0}}, in which population relocation results from the expansion of a saturated region. For this limiting model, we established several fundamental mathematical properties, including a comparison principle, a characterization of traveling wave solutions, and a precise determination of the asymptotic spreading speed.

\medskip
These results demonstrate that \eqref{eq:main0} is a sound mathematical model that yields dynamics consistent with experimental observations. 
It retains many of the key features typically associated with reaction-diffusion equations. We anticipate that additional properties not addressed here can also be extended to this framework. These include the stability of traveling wave profiles or the refined characterizations of front propagation, such as results that further specify the precise front location (e.g. Bramson-type shifts). 
Furthermore, the model exhibits some new properties which are natural for the phenomena under consideration, such as monotonicity in time or segmentation in three regions.

While we have chosen to focus here on the singular model \eqref{eq:main0}, 
the same type of questions arise for the general model \eqref{eq:main1}-\eqref{eq:L}.
Proving the existence of traveling waves for that model is actually more delicate. When working with the singular model \eqref{eq:main0}, a crucial simplification occurs as the convolution $K*{\mathbb 1}_{u=1}$ depends on the position of the front but not on the   function $u$.
In addition, it is not clear that this general model satisfies a comparison principle, so the existence of traveling wave solutions does not immediately imply the asymptotic speed of propagation property.

The basic mechanisms modeled by \eqref{eq:main1} or \eqref{eq:main0} are fairly general and apply to other settings as well where a population extends its range under the effect of congestion owing to a lack of available space rather than resources. This congestion pushes some individuals to  occupy  nearby, less constricted locations.
Potential applications include for instance ecological invasions 
or the  early dispersal of modern humans across southern Asia and Europe (where competition for space is linked to the availability of cultivable land and habitable living space).

These examples lead us to a slightly more general form of the equation:
\begin{equation}\label{eq:0}
\pa_t u  =  \Big[g(u) + \rho(u)K*{\mathbb 1}_{\{u=1\}} \Big] {\mathbb 1}_{\{u<1\}} .
\end{equation}
The function $g(u)$ represents the growth rate in the absence of competition for space (e.g. $g(u)=ru$ if resources are abundant).
The term $\rho(u) K*{\mathbb 1}_{\{u=1\}}$ describes growth at position $x$ due to the individuals who relocated due to a nearby saturate region.

It is  not difficult to extend our analysis to \eqref{eq:0}.
The existence and  uniqueness of a Lipschitz, monotone-in-time solution can be proved by proceeding as above as long as $g$ and $\rho $ are non-negative Lipschitz functions.
Spreading properties  require the additional conditions
$$ g(0)=0, \quad g(1)>0, \quad \rho (0)>0.$$
In particular, the condition $g(0)=0$ ensures the finite speed of propagation of the support (and thus the existence of semicompactly supported traveling wave solutions).
The condition $g(1)> 0$  is needed to reach the saturation level $u=1$ in finite time, which together with the condition $\rho (0)> 0$ is required to start the spatial expansion of the population. These last two conditions give a lower bound on the minimal speed of the traveling wave $c^*$.

\medskip

\paragraph{\bf Spatial growth of cities}\label{sprawl}

Equation \eqref{eq:0} provides a natural framework for the modeling of the long-term spatial expansion of cities and land-use transitions.
Indeed, the lack of space availability (congestion) plays a key role in 
the development of new urban area in previously undeveloped land and nonlocal (but finite range) jumps are natural as developers and city planners look to expand the city's footprint.

We refer to \cite{MB25} and the many references therein for a detailed discussion of various mathematical models of urban growth. 
A simple model \cite{BattySun1999} for land-transitions classifies   land into three main categories: (i) Vacant Land (rural or undeveloped land available for future constructions), (ii) New Development (area in transition), and (iii) Established Development (mature city centers where new constructions are no longer possible).
As shown in Proposition \ref{prop:support}, our model naturally describes
these three regions (respectively $\{u=0\}$, $\{0<u<1\}$ and $\{u=1\}$). Indeed, the saturated region where $u=1$  corresponds to the mature areas of the city.

We point out that in this framework, one should think of $t$ as a parameter measuring the size of the population (rather than the actual time) and $u(t,x)$ as the density of housing units at some location $x$ (or the normalized number of people who can reside at location $x$).
As the population grows (that is as $t$ increases), Equation \eqref{eq:0}  describes the addition of new constructions. 
When $0<u(t,x)<1$, there is a natural rate of construction $g(u)$ due to increased desirability: people moving to the neighborhood lead to improvement in the local infrastructure, school, commerce etc. which in turn incentivizes further local construction.

The saturated core where $u=1$ is viewed here as the inner city - where all services are readily available. A model where that threshold is different from the constraint might also be considered. Such neighborhood are the most attractive, and since further local construction is limited, it leads to new developments in surrounding (including possibly currently empty) areas.

The model can then be further specified to take into account various features associated with the growth of cities.
In particular, the total number of housing units that can be built at each location is limited by several considerations (e.g. land availability, zoning laws, environmental constraints etc.). The model we considered in this paper involved the constraint $u(t,x)\leq 1$ for simplicity.  We can extend the model by taking $u(t,x)\leq H(x)$ where $H(x)$ is a given non-negative function. Regions where $H(x)=0$ correspond to non constructible land.

We leave it to further studies to explore more in depth the relevance of models of type~\eqref{eq:0}
to represent urban growth dynamics.

\medskip

\bibliographystyle{plain}
\bibliography{saturation}

\end{document}